\documentclass[LoT,LoF,LoS,Appendices]{cmuthesis}
\usepackage{lipsum}
\usepackage{tikz,array}

\usepackage{comment}

\usepackage[all]{xy}


\theoremstyle{remark} 
\newtheorem{thm}[equation]{Theorem}
\newtheorem*{thm*}{Theorem}
\newtheorem{prop}[equation]{Proposition}
\newtheorem*{prop*}{Proposition}
\newtheorem{lem}[equation]{Lemma}
\newtheorem*{lem*}{Lemma}
\newtheorem{cor}[equation]{Corollary}
\newtheorem*{cor*}{Corollary}
\newtheorem*{conj*}{Conjecture}
\newtheorem*{fact*}{Fact}

\theoremstyle{remark} 
\newtheorem{defn}[equation]{Definition}
\newtheorem*{defn*}{Definition}
\newtheorem{ex}[equation]{Example}
\newtheorem*{ex*}{Example}
\newtheorem*{qstn*}{Question}
\newtheorem*{rem}{Remark}

\newenvironment{amatrix}[1]{%
  \left(\begin{array}{@{}*{#1}{c}|c@{}}
}{%
  \end{array}\right)
}

\newcommand{\ra}{\rightarrow}
\newcommand{\ora}{\overrightarrow}
\newcommand{\ol}{\overline}

\newcommand{\gf}{\mathfrak{g}}
\newcommand{\cx}{{\mathbb{C}}}

\newcommand{\rl}{{\mathbb{R}}}

\newcommand{\N}{\mathbb{N}}

\newcommand{\ac}{\mathcal A}

\newcommand{\dc}{\mathcal D}

\newcommand{\ci}{{\mathcal C}^\infty}

\author[Isaac Matthew Cinzori]{Isaac M. Cinzori}
\title{The Structure of the Internal Tangent Space to a Point of the Orbit Space of a Manifold under a Proper Lie Group Action}
\Type{thesis}
\Department{Mathematics}
\Degree{Master of Arts}
\Month{August}
\Year{2025}

\cleardoublepage 
\fancypagestyle{plain}{ 
  \fancyhf{}            
  \fancyhead[L]{THE STRUCTURE OF THE INTERNAL TANGENT SPACE}
  \fancyhead[R]{\thepage}
}
\begin{document}

\maketitlepage

\copyr

\begin{dedication}
 To my loving wife, Maria.\par
 And my adored parents and siblings: Aaron, Clare, Daniel, and Gwendolyn.
\end{dedication}

\begin{acknowledgements}
 I'd like to extend my heartfelt thanks to my committee members. First, to Meera Mainkar, who devoted time to my work in spite of her multiple duties as chair and graduate coordinator during various periods of my research. Her diligence in dedicating time to me and my work is deeply appreciated.

 Second, to Lisa DeMeyer, whose care in learning the notions of diffeology in order to serve on my committee and subsequent careful reading of my work led to many improvements, especially in the initial exposition. I am also grateful for the detailed, and timeless, writing advice she provided me, even if I have not yet perfected its implementation.

 Thanks are also due, of course, to my committee chair and advisor Jordan Watts. Jordan devoted a large portion of two years to me, and his continued investment in my development, work, and writing will stand for me as a model of generosity. I'm particularly grateful for the egalitarian spirit of his advising, which allowed me to grow as a mathematician while also spreading my wings in my own way (hopefully without ruffling his feathers too much).
\end{acknowledgements}

\begin{abstract}
A diffeological space is a set equipped with a smooth structure, known as a diffeology, which allows us to extend certain notions from manifolds to these more general spaces. We study a generalized notion of tangent space to a point of a manifold, namely the internal tangent space to a point of a diffeological space. In particular, we study these internal tangent spaces when the diffeological space in question is the orbit space of a manifold acted upon by a proper Lie group action. We provide a useful description for an arbitrary internal tangent space to a point of such an orbit space and then, in the culmination of our work, show that the internal tangent space to a point of an orbit space, viewed as a diffeological space, is isomorphic to the stratified tangent space to the same point, when the orbit space is viewed as a stratified space with the well-known orbit type stratification.
\end{abstract}


\startcontent

\chapter[Introduction]{\textbf{The Structure of the Internal Tangent Space to a Point of the Orbit Space of a Manifold under a Proper Lie Group Action}}\label{sec_intro}

A diffeological space is a set equipped with a smooth structure, known as a diffeology (see Definition \ref{diff def}), which usefully generalizes certain notions of traditional smooth structure on a manifold. All manifolds are diffeological spaces, with a standard and well-understood diffeology, but there are many important diffeological spaces which are not manifolds, including the orbit spaces of manifolds acted upon by Lie group actions (when the action is not free) as well as more exotic but still important examples such as the irrational torus (see \cite{diff}, exercise $31$, p. $31$).

In this paper, we focus on one notion which can be generalized from manifolds to spaces with a diffeological smooth structure: that of the tangent space to a point. On diffeological spaces, there are several notions of tangent space which extend the usual concept. In this work, we employ the notion of internal tangent space to a point, a (possibly infinite-dimensional) vector space first introduced by Hector in \cite{hector} and subsequently expanded in the works \cite{hector_follow_1} and \cite{hector_follow_2}. However, we will draw our exposition from \cite{tan_space_and_bundle}, which corrects and expands upon the earlier work. The internal tangent space to a point is given, in our work, in Definition \ref{def 3.1}.

There are many expanded notions of smooth structure which apply even when the space in question is not a manifold (see \cite{stacey} for an exposition of some of them). The idea of diffeology first formally appeared in the works of Souriau, including \cite{souriau_1} and \cite{souriau_2}, and relates to earlier work by Chen, especially \cite{chen_1} and \cite{chen_2}. The primary reference for diffeology is \cite{diff}.

In this paper, we restrict our attention to the orbit spaces of manifolds which are acted upon by proper Lie group actions. That is, we study the structure of the internal tangent space to a point of such an orbit space. These spaces are diffeological spaces with a well-understood quotient diffeology, which arises from the standard diffeology on the associated manifold.

The structure of this paper is as follows. In chapter \ref{sec_prereqs}, we carefully present the necessary background material on diffeological spaces and the internal tangent space to a point of a diffeological space, and present examples where we determine the internal tangent spaces of several orbit spaces explicitly. The results from this chapter are drawn from \cite{tan_space_and_bundle}, but the examples--although well-known--are presented in detail here.

After this, in chapter \ref{sec_lie_theory} we present the necessary theory of Lie groups and proper Lie group actions which we will employ in the remainder of the work, as well as fixing the necessary notation. This material is drawn primarily from \cite{dk_lie} and the results are well-known (although Proposition \ref{slice to isotropy} is usually not laid out explicitly).

Chapter \ref{sec_its_structure} shows that the internal tangent space to an arbitrary point of the orbit space of a manifold under a proper Lie group action is isomorphic to another internal tangent space with more structure. The primary results, Theorem \ref{main result} and Corollary \ref{main result cor}, are known to specialists in the field, but to the author's knowledge they have not yet been written down in full anywhere in the literature. We also present in detail the example of $\rl^3$ acted upon by the rotation group $SO(3)$ to show the utility of our findings.

In chapter \ref{sec_strat_stuff}, the culmination of our work, we show that the internal tangent space to a point of an orbit space acted upon by a proper Lie group action, with the orbit space viewed as a diffeological space, is isomorphic to the stratified tangent space to the same point in the orbit space, with the orbit space now viewed as a stratified space. Stratified spaces are, roughly speaking, spaces which can be partitioned into ``pieces" called strata, each of which is a manifold (hence each point has a natural associated tangent space). After briefly introducing these spaces and the natural ``orbit-type stratification" possessed by all orbit spaces of manifolds acted upon by proper Lie group actions, detailed in \cite{dk_lie} and \cite{strat}, we turn to our findings. The culminating result is Theorem \ref{ITS to STS}, which itself is due primarily to Lemmas \ref{ITS to STS lem 1} and \ref{ITS to STS lem 2} and a function, which we call the average map, presented in \cite{dk_lie} and described here in Proposition \ref{avg}. Theorem \ref{ITS to STS} is novel to our work in its entirety. We close with a reprise of our earlier examples of internal tangent spaces, now considering the stratified tangent space associated to the same point of the given orbit space.

The results given here, especially the culminating Theorem \ref{ITS to STS}, are significant because they give a greater understanding of the internal tangent space to a point in the context of the orbit spaces we are studying. In addition to the practical tool for determining the internal tangent spaces of certain orbit spaces more easily provided by Corollary \ref{main result cor}, this work may have a place in studies of diffeological analogues of other manifold notions which depend on the notion of internal tangent space, at least when the spaces studied are our relevant orbit spaces (which include, for instance, orbifolds). As examples, consider the works of \cite{applications_3} and especially \cite{applications_2} on immersions in the diffeological sense.

Further, our work can help shed light on the internal tangent bundle, analogous to the conventional manifold tangent bundle, which can be formed from the internal tangent spaces to each point of a diffeological space. It is studied in \cite{tan_space_and_bundle} and \cite{applications_1}. The internal tangent bundle itself is an example of what is termed a diffeological vector pseudo-bundle, and so enhanced knowledge of the internal tangent bundle can expand our collective understanding of these spaces as well (which are studied in, specifically, \cite{tan_space_and_bundle}, \cite{applications_1}, \cite{dvs_applications_1}, \cite{dvs_applications_2}, and \cite{dvs_applications_3}).

Lastly, we again highlight the importance of Theorem \ref{ITS to STS}. Its utility is in linking the newer concept of internal tangent space to a point to the well-understood notion of stratified tangent space to a point. In particular, the orbit spaces we study admit a well-known orbit-type stratification. Further, the stratified tangent spaces associated to a stratified space can form a stratified tangent bundle, which is also well-understood in the context of orbit-type stratified orbit spaces. Our result linking the internal and stratified tangent spaces to a point may therefore also be useful in linking the internal tangent bundle associated to an orbit space to its stratified tangent bundle, although care is required as traditionally an alternative smooth structure, known as a Sikorski smooth structure (first introduced in \cite{sikorski_1} and \cite{sikorski_2}), is placed on stratified spaces and their associated tangent bundles.

Exposition on the orbit-type stratification of the orbit space associated to a manifold acted upon by a proper Lie group action, as well as the resulting Sikorski smooth structure which can be assigned, can be found in \cite{strat}, and further relevant references include \cite{orbit_strat_1}, \cite{orbit_strat_2}, \cite{orbit_strat_3}, \cite{orbit_strat_4}, and \cite{orbit_strat_5}. While beyond the scope of this work, we believe that Theorem \ref{ITS to STS} will be useful in further investigations of how the internal tangent bundle associated to an orbit space relates to its stratified tangent bundle.

In this paper, all manifolds are assumed to be Hausdorff, second-countable, smooth, and without boundary, all vector spaces are assumed to be over $\rl$, and all linear maps $\rl$-linear, unless explicitly stated otherwise. The term ``diffeomorphism" refers to a conventional diffeomorphism in the manifold sense when written in chapter \ref{sec_lie_theory}, but to a diffeological diffeomorphism (see Definition \ref{diff diffeomorphism}) elsewhere, again unless mentioned otherwise.

\nocite{category_theory}
\nocite{lee_top}

\chapter[Diffeological Spaces and the Internal Tangent Space]{\textbf{Chapter II. Diffeological Spaces and the Internal Tangent Space}}\label{sec_prereqs}

We lay out here the diffeological notions used in this work. Our primary references will be \cite{tan_space_and_bundle} and \cite{diff}, but we will also reference \cite{d_top}, \cite{hector}, and both \cite{watts_paper} and \cite{watts_thesis}.

\begin{defn}[\cite{diff} Definition $1.5$]\label{diff def}
    A \textit{parametrization} of a set $X$ is any map $p:U \ra X$ where $U\subseteq \rl^n$ (for some $n\in \{0\}\cup\N$) is open in the standard topology on $\rl^n$. A \textit{diffeological space} $X$ is a nonempty set along with a specified set of paramatrizations of $X$, denoted $\dc_X$, which satisfy three conditions:
    \begin{enumerate}
        \item (covering) Every constant parametrization of $X$ of the form $p : \rl^n \ra X$ with $p(\rl^n)=\{x\}$ lies in $\dc_X$, for every $n\in \{0\}\cup\N$ and every $x\in X$.
        \item (smooth compatibility) For every parametrization $p: U \ra X$ of $\dc_X$, and every open subset $V$ of $\rl^m$ (for $m\in \{0\}\cup\N$), and every smooth map $F : V \ra U$, it is the case that $p\circ F\in \dc_X$.
        \item (locality) If $p:U \ra X$ is a parametrization of $X$ such that for every $u\in U$ there is an open neighborhood $V\subseteq U$ of $u$ for which the restricted map $p|_V$ lies in $\dc_X$, then $p\in \dc_X$.
    \end{enumerate}
    
    The collection $\dc_X$ is called a \textit{diffeology} on $X$, and the parametrizations in the diffeology are called \textit{plots} of the diffeological space $X$.
\end{defn}

\begin{defn}[\cite{diff} Definition $1.14$]\label{diff diffeomorphism}
    A function $f : X\ra Y$ between diffeological spaces is called \textit{smooth} if whenever $p : U\ra X$ is a plot of $X$, then $f \circ p$ is a plot of $Y$. If a given smooth function has an inverse that is also smooth, it is called a \textit{diffeomorphism}.
\end{defn}


Given two diffeological spaces $X$ and $Y$, the set of smooth maps between them is denoted $\ci(X,Y)$.

\begin{ex}
    Every smooth manifold $M$ can be viewed as a diffeological space with the diffeology consisting of all smooth maps of manifolds of the form $p:U\ra M$ (where again $U\subseteq \rl^n$, $n\in \{0\} \cup \N$, is an open subset in the standard topology). The first requirement of Definition \ref{diff def} is satisfied because all constant maps between manifolds are smooth, the second because compositions of smooth maps between manifolds remain smooth, and the third because smoothness of maps between manifolds is a local concept. We call this collection of plots the \textit{standard diffeology} on the manifold $M$. In this case, the diffeologically smooth maps (both from $M$ to $M$ and from $M$ to another manifold $N$) are precisely those which are smooth in the usual sense (it is helpful to use Boman's Theorem here; see \cite{boman}). We assume this diffeology is the one given to all manifolds, unless noted.
\end{ex}

We term the smallest diffeology on a set $X$ which contans a set of parametrizations $\ac = \{ p_i:U_i \ra X \}_{i\in I}$ the diffeology \textit{generated} by the collection $\ac$. This diffeology is comprised of precisely the parametrizations $p:U \ra X$ that either locally factor through the given functions via smooth maps or are locally constant (only locally, due to conditions $2$ and $3$ in Definition \ref{diff def}).

\begin{ex}\label{ex wire}
    Fix $n\in \N$. The space $X=\rl^n$ is a manifold, hence it can be viewed as a diffeological space equipped with its standard diffeology of smooth maps $p:U\ra \rl^n$ (for $U$ open in $\rl^m$, $m\in \{0\} \cup \N$). However, there is another diffeology that can be placed on $X$. The \textit{wire diffeology} on $X=\rl^n$ is the diffeology generated by the set of \textit{smooth} parametrizations of $\rl^n$ with domain $\rl$; that is, by the set $\{p:\rl \ra \rl^n | \text{ $p$ is a smooth parametrization of $\rl^n$}\}$. It consists of all parametrizations of $\rl^n$ which locally factor through smooth one-dimensional parametrizations. $X$ is not diffeomorphic to $\rl^n$ with the standard diffeology when $n\ge2$, as not all plots in the standard diffeology locally factor through a smooth curve in these cases. For instance, the identity map $id:\rl^n \ra \rl^n$ is not a plot unless $n=1$, as this map is not constant and it does not factor even locally as $id=p\circ F$ (for $p : U\subseteq \rl \ra \rl^n$ a plot in the wire diffeology and $F : V \subseteq \rl^m \ra U$ a smooth map, with $U$ and $V$ open sets) when $n\ge 2$.
\end{ex}

\begin{ex}
    Any nonempty set $X$ admits at least two (usually distinct) diffeologies:  the \textit{indiscrete diffeology} consisting of all parametrizations $p:U\ra X$ of $X$ and the \textit{discrete diffeology} consisting of only the parametrizations of $X$ which are locally constant. The first of these immediately satisfies the conditions required to be a diffeology given in Definition \ref{diff def}; the second is the diffeology on the space $X$ generated by the constant parametrizations of $X$.
\end{ex}

\begin{rem}
    The usage of ``indiscrete" and ``discrete" in the naming of diffeologies is chosen so that the $D$-topology on a nonempty set $X$ (introduced in Definition \ref{d-top def}) is the indiscrete topology in the usual sense when $X$ is equipped with its indiscrete diffeology and the discrete topology in the usual sense when $X$ is equipped with its discrete diffeology.
\end{rem}

Suppose there is an equivalence relation $\sim$ on a given diffeological space $X$. The smallest diffeology on the quotient set $X/\!\sim$ making the quotient map $\pi : X \twoheadrightarrow X/\!\sim$ smooth is called the \textit{quotient diffeology} of $X/\!\sim$, which makes it into a diffeological space. $\dc_{X/\sim}$ consists of all parametrizations $\ol{p}:U \ra X/\!\sim$ that locally factor through the quotient map $\pi$. That is, locally $\ol{p}=\pi \circ p$ for $p\in \dc_X$.

\begin{ex}
    The orthogonal group of dimension $0$, $O(1)$, acts on $\rl$ via the mapping $\pm1 \cdot x = \pm x$. The orbit space $\rl / O(1)$ is a diffeological space with quotient diffeology consisting of all paramterizations $\overline{p}:U_{\overline{p}} \ra \rl / O(1)$ which locally factor as $\overline{p} = \pi \circ p$, for $p:U_{\ol{p}} \ra \rl$ a smooth map in the usual sense (that is, $p$ lies in the standard diffeology on $\rl$; when the space being quotiented is a manifold, it is not a typo that $\ol{p}$ and $p$ share the same domain).

    Indeed, the requirement that $\pi$ be smooth means that all parametrizations of the form $\pi \circ p$ must be included in the quotient diffeology on $\rl / O(1)$, for $p$ any plot in the standard diffeology on $\rl$. It then follows that all parametrizations on $\rl / O(1)$ that only locally factor as such must lie in the quotient diffeology as well, due to the third condition in Definition \ref{diff def}.
\end{ex}

Suppose we have a subset $A$ of a diffeological space $Y$. The largest diffeology on $A$ for which the inclusion map $\iota : A \hookrightarrow Y$ is smooth is called the \textit{subset diffeology} of $A$. By definition, it consists of precisely the parametrizations $p:U \ra A$ such that $\iota \circ p$ is a plot in the diffeology $\dc_Y$ on $Y$.

\begin{ex}
    The subset $[0,\infty)$ of $\rl$ has a subset diffeology consisting of all the parametrizations $p:U \ra [0,\infty)$ such that $\iota \circ p$ is a plot of the standard diffeology on $\rl$.

    Indeed, all constant mappings into $[0,\infty)$ lie in the standard diffeology on $\rl$ after composing with $\iota : [0,\infty)\ra \rl$. Further, given a smooth map $F : V \ra U$ between Euclidean open subsets and a parametrization $p : U \ra [0,\infty)$ such that $\iota \circ p$ is a smooth map, $\iota \circ p \circ F$ remains smooth. Lastly, if we have a parametrization $p : U \ra [0,\infty)$ such that locally $\iota \circ p : U \ra \rl$ is a smooth map, then because smoothness of maps between manifolds is local it follows that $\iota \circ p$ is smooth globally. All three conditions in Definition \ref{diff def} are satisfied.
\end{ex}

\begin{rem}
    Note that $\rl / O(1)$ with the quotient diffeology is not diffeomorphic to $[0,\infty)$ with the subset diffeology. See \cite{H1HSub} for a plot only in the latter diffeology.
\end{rem}

All diffeological spaces have a natural topology.

\begin{defn}\label{d-top def}
    If $X$ is a diffeological space with diffeology $\dc_X$, then a subset $A$ of $X$ is open in the \textit{$D$-topology} of $X$ if and only if $p^{-1}(A)$ is open for each $p\in \dc_X$ (the domain of each plot is equipped with the standard topology). An open set in this topology will be referred to as \textit{$D$-open}, when clarity is needed.
\end{defn}

Because inverse images commute with unions and intersections, and the domains of all plot maps are open sets, it follows that the $D$-topology is a topology.

When $X$ and $Y$ are two diffeological spaces equipped with their respective $D$-topologies, any smooth map $f:X\ra Y$ is also continuous. To see this, note that if $U_Y$ is a $D$-open set in $Y$, then for every $p_Y:V \ra Y$ in $\dc_Y$, $p_Y^{-1}(U_Y)$ is open. Consider now $f^{-1}(U_Y)$. Because for any plot $p_X : U \ra X$ in $\dc_X$, $f\circ p$ is a plot in $\dc_Y$, it follows that $p_X^{-1}(f^{-1}(U_Y))=(f\circ p_X)^{-1}(U_Y)$ is open in $U$, hence $f^{-1}(U_Y)$ is $D$-open in $X$, meaning that $f$ is continuous.

\begin{ex}
    The $D$-topology on a smooth manifold equipped with the standard diffeology coincides with the usual topology. Indeed, if a set $U$ of an $n$-dimensional manifold $M$ is $D$-open, then the inverse image of $U$ under a chart $\varphi$ centered about any point $x\in U$ (which is a plot) is open in $\rl^n$ equipped with the standard topology. However, because all chart maps are diffeomorphisms from an open subset of the topological space $M$ to Euclidean space with the standard topology, it follows--after restricting as necessary--that each point $x\in U$ is contained in an open neighborhood lying in $U$, hence $U$ is an open set in the original topology on the manifold.
    
    Conversely, if $U'$ is an open set in the topology on $M$, then because any plot $p$ in $\dc_M$ (the standard diffeology) is, as a smooth map of manifolds, also continuous we have that $p^{-1}(U')$ is open. Thus $U'$ is $D$-open. Therefore, a set is open in the topology assigned to $M$ as a manifold if and only if it is $D$-open.
\end{ex}

Let $X$ be a diffeological space and $Y$ a quotient set of $X$. There are a priori two topologies relevant to this discussion which could be given to $Y$. First, the $D$-topology of the quotient diffeology on $Y$. Second, the quotient topology of the $D$-topology on $X$. However, these two topologies are actually the same (see \cite{d_top} $3.3$ or \cite{diff} $2.12$).

\begin{ex}\label{ex d-top}
    When we have a manifold $M$ and consider a quotient set of $M$, the $D$-topology on the quotient is just the usual quotient topology (since the $D$-topology on $M$ is in fact the standard topology). In particular, the $D$-topology on $\rl / O(1)$ is just the quotient topology arising from the standard topology on $\rl$ and induced by the quotient map $\pi : \rl \ra \rl / O(1)$.
\end{ex}

The internal tangent space to a point of a diffeological space is an extension of the notion of tangent space to a point of a manifold (as shown in Example \ref{ex man}). It can be defined as a colimit, but here shall be given concretely.

A \textit{plot pointing to $x$} of a diffeological space $X$ is a parametrization $p:U \ra X$ in $\dc_X$ with a connected domain which contains $0$ such that $p(0)=x$ (for some specified $x\in X$). These will be called \textit{pointed plots} when the $x$ being pointed to is clear, and are written $p : (U, 0) \rightarrow (X, x)$ when greater clarity is needed.
 
\begin{defn}[\cite{tan_space_and_bundle} Definition $3.1$ and \cite{hector}] \label{def 3.1}
    Let $x$ be an arbitrary point of a diffeological space $X$. The \textit{internal tangent space to a point x} in $X$, denoted $T_x(X)$, is the quotient vector space $F/R$, where $F$ is the sum
    $$F=\bigoplus_{p~ : ~ (U_p,0) \ra (X,x), ~ p\in \dc_X}T_0(U_p)$$
    \noindent indexed by the collection of pointed plots in $\dc_X$ and $R\subseteq F$ is the space of vectors spanned by those of the form $(p,v) - (q,g_*(v))$ for which there is a germ of smooth maps $g : (U_p,0) \ra (U_q, 0)$ such that for the pointed plots $p : (U_p, 0) \rightarrow (X, x)$ and $q : (U_q, 0) \rightarrow (X, x)$ we have $p=q\circ g$ as germs. That is,
    $$R=\text{Span}\{(p,v) - (q,g_*(v))\}$$
    \noindent for which the following commutative diagram holds locally:
    $$
    \xymatrix{
    (U_{p},0) \ar[rr]^{g} \ar[dr]_{p} & & (U_q,0) \ar[dl]^{q} \\
    & (X,x) &
    }
    $$
    Here, $(p, v)$ denotes $v\in T_0(U_p)$ (a tangent space in the usual sense) in the summand of $F$ indexed by $p$ and $(q,g_*(v))$ denotes $g_*(v)\in T_0(U_q)$ in the summand of $F$ indexed by $q$. An element of $R$ of the form $(p,v) - (q,g_*(v))$ will be termed a \textit{basic relation} with $U_p$ the \textit{domain of the relation}.

    When two elements $v_1$ and $v_2$ are related in $F/R$, we will write $v_1 \sim v_2$ as shorthand. This means $v_1-v_2 \in R$.
\end{defn}

\begin{rem}
    We sometimes omit the parentheses and index plots to simplify expressions. The element $(p,v)\in F/R$ is also written as $p_*(v)$ in the literature, though the plot $p$ may not have a differential in the usual sense.
\end{rem}

\begin{ex}\label{ex man}
Let $M$ be a manifold of dimension $n$. We determine the internal tangent space of an arbitrary point $x\in M$. It is known in this case that the internal tangent space aligns with the tangent space (which is isomorphic--in the usual sense--to $\rl^n$).

Let $x$ be an arbitrary point of the given manifold $M$. As laid out in Definition \ref{def 3.1}, $T_x(M)$ is the quotient vector space $F/R$, where $F=\oplus_pT_0(U_p)$ with sum indexed over the pointed plots $p:(U_p,0)\rightarrow (M,x)$ and $R$ is the span of the vectors of the form $(p,v) - (q,g_*(v))$ for which there is a germ of smooth maps $g : (U_p,0) \ra (U_q, 0)$ such that for the pointed plots $p : (U_p, 0) \rightarrow (X, x)$ and $q : (U_q, 0) \rightarrow (X, x)$ we have $p=q\circ g$ as germs.

Now, in particular, let $(U,\varphi)$ be a chart centered about $x$ (guaranteed as $M$ is a manifold) and let $p : U_p \ra M$ be any plot in the standard diffeology on $M$ mapping $0$ to $x$. In this instance, $p$ is a smooth map in the conventional sense. We have the following diagram:

$$
\xymatrix{
U_{p} \ar[rr]^{\varphi \circ p} \ar[dr]_{p} & & \varphi(U) \ar[dl]^{\varphi^{-1}} \\
& M &
}
$$

\noindent where $p=\varphi^{-1} \circ (\varphi \circ p)$ as germs about $0$, meaning $(p,v) \sim (\varphi^{-1}, (\varphi \circ p)_*(v))$ in $F/R$.

Now, suppose we have an arbitrary element $w$ of $F/R$. Then $w$ consists of an element of $F=\oplus_pT_0(U_p)$ with only finitely many nonzero components $v_i \in T_0(U_{p_i})$, modulo the relations in $R$. In fact, due to these relations it is the case that for each nonzero component $v_i$, we have $w \sim w + [-(p_i,v_i) + (\varphi^{-1}, (\varphi \circ p_i)_*(v_i) )]$.

Therefore, the above work shows that any element of $F/R$ is equivalent to an element whose only nontrivial summand term (if any) lies in $T_0(\varphi(U))$ and so we obtain a surjective homomorphism mapping $T_0(\varphi(U)) \twoheadrightarrow F/R$. Since $T_0(\varphi(U)) \cong \rl^n$, it only remains to show our mapping is injective. To do so, it is sufficient to show that the distinct elements in the summand $T_0(\varphi(U))$ of $F$ remain distinct in $F/R$.

Indeed, if two elements $v$ and $w$ of $F$ whose only nontrivial terms were contained in $T_0(\varphi(U))$ were equivalent (that is, $v\sim w$), then we would have $v-w\in R$ (though not necessarily a basic relation). That is,

\begin{align}
    (\varphi^{-1},v)-(\varphi^{-1},w)=(\varphi^{-1},v - w)=\sum_i c_i[(p_i,v_i)-(q_i,g^i_*v_i)]\label{man cancellation}
\end{align}

\noindent for $c_i\in \rl$, $p_i,q_i$ pointed plots mapping $0$ to $x$, each $g^i$ a germ of smooth maps such that $p_i = q_i \circ g^i$ as germs, and $v_i\in T_0(U_{p_i})$. Note the above equality also holds in $F$, prior to modding out by $R$, and we will think of it in this sense below.

Because the equality above holds in $F$, note that because the left-hand side consists solely of elements from the summand corresponding to the plot $\varphi^{-1}$, any elements on the right-hand side corresponding to other summands (i.e., plots $p_i$ or $q_i$ not equal to $\varphi^{-1}$) must cancel completely with solely other elements on the right-hand side or be zero. Now, expand the above equation to:

\begin{align*}
    \sum_i c_i[(p_i,v_i)-(q_i,g^i_*v_i)] &= \sum_i c_i[(p_i,v_i)-(\varphi^{-1},(\varphi \circ p_{i})_*v_i)] \\
    &+ \sum_ic_i[(\varphi^{-1},(\varphi \circ p_{i})_*v_i)-(\varphi^{-1},(\varphi \circ q_{i})_*g^i_*v_i)] \\
    &+\sum_ic_i[(\varphi^{-1},(\varphi \circ q_{i})_*g^i_*v_i)-(q_i,g^i_*v_i)]
\end{align*}

\noindent and then rewrite the right-hand side as

\begin{align}
    & \quad \sum_ic_i[(\varphi^{-1},(\varphi \circ p_{i})_*v_i)-(\varphi^{-1},(\varphi \circ q_{i})_*g^i_*v_i)]\label{man regrouped 1.1} \\
    &+ \sum_i c_i[(p_i,v_i)-(\varphi^{-1},(\varphi \circ p_{i})_*v_i)+(\varphi^{-1},(\varphi \circ q_{i})_*g^i_*v_i)-(q_i,g^i_*v_i)]\label{man regrouped 1.2}
\end{align}

Now, regroup the related pairs of \ref{man regrouped 1.2} not by index, but rather by plot. That is, whenever $p_i$ agrees with some other $p_j$ or $q_k$ (and likewise for plots $q_i$), group all such pairs $(p_i,v_i)-(\varphi^{-1},(\varphi \circ p_{i})_*v_i)$ or $(\varphi^{-1},(\varphi \circ q_{i})_*g^i_*v_i)-(q_i,g^i_*v_i)$ together. With this organization, where we label groups of the same plot using only $p's$ (no longer $q's$), expression \ref{man regrouped 1.1}/\ref{man regrouped 1.2} becomes:

\begin{align*}
    (\varphi^{-1}, v - w) &= \sum_ic_i[(\varphi^{-1},(\varphi \circ p_{i})_*v_i)-(\varphi^{-1},(\varphi \circ q_{i})_*g^i_*v_i)] \\
    &+ \sum_{j=1}^{m_1}c^1_j[(p_1,v^1_j)-(\varphi^{-1},(\varphi \circ p_{1})_*v^1_j)] \\
    & \quad \quad +\sum_{k=1}^{n_1} c^1_k[(\varphi^{-1},(\varphi \circ p_{1})_*(g^{1,k})_*v^1_k)-(p_1,(g^{1,k})_*v^1_k)] \\
    & \vdots \\
    &+ \sum_{j=1}^{m_N} c^N_j[(p_N,v^N_j)-(\varphi^{-1},(\varphi \circ p_{N})_*v^N_j)] \\
    & \quad \quad +\sum_{k=1}^{n_N} c^N_k[(\varphi^{-1},(\varphi \circ p_{N})_*(g^{N,k})_*v^N_k)-(p_N,(g^{N,k})_*v^N_k)]
\end{align*}

\noindent where it is possible that $m_i$ or $n_i$ may be zero. We now proceed to show the above collection of sums equals zero. First, note that because we have $p_i = q_i \circ g^i$ as germs about $0$, it is the case that $\varphi_* p_{i*}=\varphi_* q_{i*} g^i_*$. Therefore

$$\sum_ic_i[(\varphi^{-1},(\varphi \circ p_{i})_*v_i)-(\varphi^{-1},(\varphi \circ q_{i})_*g^i_*v_i)] = \sum_ic_i[(\varphi^{-1},(\varphi \circ p_{i})_*v_i)-(\varphi^{-1},(\varphi \circ p_{i})_*v_i)] = 0$$

\noindent and so this entire sum (the first part of our expanded expression) can be eliminated from the expression for $(\varphi^{-1},v-w)$. This leaves us with

\begin{align*}
    (\varphi^{-1}, v - w) &= \sum_{j=1}^{m_1}c^1_j[(p_1,v^1_j)-(\varphi^{-1},(\varphi \circ p_{1})_*v^1_j)] \\
    & \quad \quad +\sum_{k=1}^{n_1} c^1_k[(\varphi^{-1},(\varphi \circ p_{1})_*(g^{1,k})_*v^1_k)-(p_1,(g^{1,k})_*v^1_k)] \\
    & \vdots \\
    &+ \sum_{j=1}^{m_N} c^N_j[(p_N,v^N_j)-(\varphi^{-1},(\varphi \circ p_{N})_*v^N_j)] \\
    & \quad \quad +\sum_{k=1}^{n_N} c^N_k[(\varphi^{-1},(\varphi \circ p_{N})_*(g^{N,k})_*v^N_k)-(p_N,(g^{N,k})_*v^N_k)]
\end{align*}

Now, observe that if $p_\ell = \varphi^{-1}$ for one of our remaining $N$ groupings, then said grouping takes the form

$$\sum_{j=1}^{m_\ell}c^\ell_j[(\varphi^{-1},v^\ell_j)-(\varphi^{-1},(\varphi \circ \varphi^{-1})_*v^\ell_j)] +\sum_{k=1}^{n_\ell} c^\ell_k[(\varphi^{-1},(\varphi \circ \varphi^{-1})_*(g^{\ell,k})_*v^\ell_k)-(\varphi^{-1},(g^{\ell,k})_*v^\ell_k)]=0$$

Therefore, such a group can be eliminated. We have so far shown that for any $v$ and $w$ equivalent in the summand $T_0(\varphi(U))$ of $F$ (where $F/R$ is the internal tangent space to $x$ in $M$) can be expressed as

\begin{align*}
    (\varphi^{-1}, v - w) &= \sum_{j=1}^{m_1}c^1_j[(p_1,v^1_j)-(\varphi^{-1},(\varphi \circ p_{1})_*v^1_j)] \\
    & \quad \quad +\sum_{k=1}^{n_1} c^1_k[(\varphi^{-1},(\varphi \circ p_{1})_*(g^{1,k})_*v^1_k)-(p_1,(g^{1,k})_*v^1_k)] \\
    & \vdots \\
    &+ \sum_{j=1}^{m_N} c^N_j[(p_N,v^N_j)-(\varphi^{-1},(\varphi \circ p_{N})_*v^N_j)] \\
    & \quad \quad +\sum_{k=1}^{n_N} c^N_k[(\varphi^{-1},(\varphi \circ p_{N})_*(g^{N,k})_*v^N_k)-(p_N,(g^{N,k})_*v^N_k)]
\end{align*}

\noindent where $p_\ell \neq \varphi^{-1}$. We now show that these remaining terms can also be eliminated. Indeed, for such a grouping associated to a plot $p_\ell$, namely

\begin{align*}
& \quad \sum_{j=1}^{m_\ell}c^\ell_j[(p_\ell,v^\ell_j)-(\varphi^{-1},(\varphi \circ p_{\ell})_*v^\ell_j)] +\sum_{k=1}^{n_\ell} c^\ell_k[(\varphi^{-1},(\varphi \circ p_{\ell})_*(g^{\ell,k})_*v^\ell_k)-(p_\ell,(g^{\ell,k})_*v^\ell_k)] \\
&= \left(p_\ell,\sum_{j=1}^{m_\ell} c^\ell_jv^\ell_j - \sum_{k=1}^{n_\ell}c^\ell_k(g^{\ell,k})_*v^\ell_k\right) + \left(\varphi^{-1},-\sum_{j=1}^{m_\ell} c^\ell_j(\varphi\circ p_{\ell})_*v^\ell_j + \sum_{k=1}^{n_\ell}c^\ell_k(\varphi\circ p_{\ell})_*(g^{\ell,k})_*v^\ell_k\right)
\end{align*}

\noindent we must have that the first term in the second line above satisfies

\begin{align}
    \sum_{j=1}^{m_\ell}c^\ell_j(p_\ell,v^\ell_j) - \sum_{k=1}^{n_\ell}c^\ell_k(p_\ell,(g^{\ell,k})_*v^\ell_k) = \left(p_\ell,\sum_{j=1}^{m_\ell} c^\ell_jv^\ell_j - \sum_{k=1}^{n_\ell}c^\ell_k(g^{\ell,k})_*v^\ell_k\right)=0\label{man regrouped 2}
\end{align}

\noindent because $p_\ell\neq \varphi^{-1}$ and, as noted initially (see the paragraphs following \ref{man cancellation}), such terms must cancel completely (and only with other terms indexed by the same plot $p_\ell$). Now, as all $v^\ell_j$ and $(g^{\ell,k})_*v^\ell_k$ lie in $T_0(U_{p_\ell})$ for all $j,k$, we can apply the differential $\varphi_*p_{\ell*}$ to this collection of vectors to observe that

\begin{align}
    0 &= -1\cdot \varphi_* p_{\ell*} \left[\sum_{j=1}^{m_\ell} c^\ell_jv^\ell_j - \sum_{k=1}^{n_\ell}c^\ell_k(g^{\ell,k})_*v^\ell_k \right] \\
    &= -\sum_{j=1}^{m_\ell} c^\ell_j\varphi_*p_{\ell*}v^\ell_j + \sum_{k=1}^{n_\ell}c^\ell_k\varphi_*p_{\ell*}(g^{\ell,k})_*v^\ell_k \label{man regrouped 3}
\end{align}

From the above equations \ref{man regrouped 2} and \ref{man regrouped 3}, it follows that

\begin{align*}
& \quad \sum_{j=1}^{m_\ell}c^\ell_j[(p_\ell,v^\ell_j)-(\varphi^{-1},(\varphi \circ p_{\ell})_*v^\ell_j)] +\sum_{k=1}^{n_\ell} c^\ell_k[(\varphi^{-1},(\varphi \circ p_{\ell})_*(g^{\ell,k})_*v^\ell_k)-(p_\ell,(g^{\ell,k})_*v^\ell_k)] \\
&= \left(p_\ell,\sum_{j=1}^{m_\ell} c^\ell_jv^\ell_j - \sum_{k=1}^{n_\ell}c^\ell_k(g^{\ell,k})_*v^\ell_k\right) + \left(\varphi^{-1},-\sum_{j=1}^{m_\ell} c^\ell_j(\varphi\circ p_{\ell})_*v^\ell_j + \sum_{k=1}^{n_\ell}c^\ell_k(\varphi\circ p_{\ell})_*(g^{\ell,k})_*v^\ell_k\right) \\
&= 0
\end{align*}

\noindent as desired. Hence each remaining term in the expression for $(\varphi^{-1},v-w)$ can be eliminated. This means that overall (in $F$, before modding out by $R$)

\begin{align*}
    (\varphi^{-1}, v - w) &= \sum_{j=1}^{m_1}c^1_j[(p_1,v^1_j)-(\varphi^{-1},(\varphi \circ p_{1})_*v^1_j)] \\
    & \quad \quad +\sum_{k=1}^{n_1} c^1_k[(\varphi^{-1},(\varphi \circ p_{1})_*(g^{1,k})_*v^1_k)-(p_1,(g^{1,k})_*v^1_k)] \\
    & \vdots \\
    &+ \sum_{j=1}^{m_N} c^N_j[(p_N,v^N_j)-(\varphi^{-1},(\varphi \circ p_{N})_*v^N_j)] \\
    & \quad \quad +\sum_{k=1}^{n_N} c^N_k[(\varphi^{-1},(\varphi \circ p_{N})_*(g^{N,k})_*v^N_k)-(p_N,(g^{N,k})_*v^N_k)] \\
    &= 0
\end{align*}

This means that $v \sim w$ if and only if $v = w$ in $T_0(\varphi(U))$, giving the conclusion. Therefore, when we have a manifold $M$ of dimension $n$, its internal tangent space at a point $x$ in $M$, given by $F/R$, is isomorphic to $T_0(\varphi(U)) \cong \rl^n$, where $(U, \varphi)$ is a chart centered about $x$ in $M$. The notion of internal tangent space at a point is identical to the notion of tangent space at a point, for a manifold.
\end{ex}

\begin{ex}\label{ex r1}
The internal tangent space to a point is defined for any diffeological space, including spaces which are not manifolds. As we are studying the orbit spaces of manifolds which are acted upon by proper Lie group actions, consider the specific case of the orbit space of the manifold $\rl$ acted on by $O(1)$, the orthogonal group of dimension $0$, equipped with the quotient diffeology. This action consists of the automorphisms $x \mapsto x$ and $x\mapsto-x$. For more details on the Lie theory involved, see the following chapter.

It is known that for $[x]\in \rl / O(1)$, $T_{[x]}(\rl / O(1)) \cong \rl$ when $[x] \ne [0]$ and $T_{[0]}(\rl / O(1))=\{0\}$ (see \cite{tan_space_and_bundle} Example $3.24$). We proceed to show both of these results, as an elucidation of our methods. Again, we employ Definition \ref{def 3.1} and view the internal tangent space to a point as the quotient $F/R$, and let $\pi:\rl \ra \rl / O(1)$ denote the projection map.

When looking at a point $[x]\ne[0]$ in $\rl / O(1)$, the argument is akin to the case when the space is a manifold. Indeed, consider a chart $(U,\varphi)$ centered about $|x|$ in $\rl$ with domain an open interval including $|x|$ but not $0$. With this setup, $\pi \circ \varphi^{-1}$ is a plot in the quotient diffeology on $\rl/O(1)$ mapping $0$ to $[x]$. Now, let $\ol{p}:U_{\ol{p}} \ra \rl /O(1)$ be any plot in the quotient diffeology on $\rl / O(1)$ mapping $0$ to $[x]$. As $\ol{p}$ lies in the quotient diffeology, we have that locally about $0$ in $U_{\ol{p}}$, $\ol{p}=\pi \circ p$ for $p$ a conventional smooth map $U_{\ol{p}} \ra \rl$ (i.e., a plot in the standard diffeology on $\rl$).

Because $\ol{p}(0)=[x]$, we may assume $p(0)=|x|$. We then have the following diagram:

$$
\xymatrix{
U_{\ol{p}} \ar[rr]^{\varphi \circ p} \ar[dr]_{\ol{p}} & & \varphi(U) \ar[dl]^{\pi \circ \varphi^{-1}} \\
& \rl/O(1) &
}
$$

\noindent where $\ol{p}=\pi \circ p=\pi \circ \varphi^{-1} \circ \varphi \circ p$ as germs, meaning $(\ol{p}, v) \sim (\pi \circ \varphi^{-1}, (\varphi \circ p)_*(v))$ in $F/R$.

We now address the question of well-definedness. Indeed, it is possible that $\ol{p}=\pi \circ p'$ as germs about $0$ for some $p'$ distinct from $p$. However, when this occurs then we must have $\pi \circ p = \pi \circ p'$ as germs about $0$, meaning that for $w$ in a neighborhood of $0$ in $U_{\ol{p}}$ we have

\begin{align}
    p(w)=\pm p'(w)\label{man well-defined}
\end{align}

However, as both $p$ and $p'$ are smooth maps in the usual sense (as elements of the standard diffeology on $\rl$) and $\ol{p}(0) = [x] \neq [0]$ we must have either $p(w) = p'(w)$ as germs or $p(w) = -p'(w)$ as germs about $0$ (i.e., there is a neighborhood of $0$ about which the sign does not change). Here we are explicitly using that $[x]\neq [0]$. Furthermore, because $\varphi$ is centered around $|x|$ and does not include $-|x|$ in its domain, we only have $\pi \circ p' = \pi \circ \varphi^{-1} \circ \varphi \circ p'$, that is, $(\ol{p},v)\sim (\pi\circ \varphi^{-1},(\varphi\circ p')_*(v))$ for some $p'$ distinct from $p$ if $p'(0)=p(0)=|x|$, which by the above implies $p=p'$ as germs about $0$. Thus, our usage of $p$ is unambiguous.

What the discussion so far means is that, as with the internal tangent space to a point of a manifold, we again obtain a surjective homomorphism mapping $T_0(\varphi(U)) \twoheadrightarrow F/R = T_{[x]}(\rl/O(1))$ when $[x]\neq [0]$ (though now the summand corresponds to plot $\pi \circ \varphi^{-1}$).  Since $T_0(\varphi(U)) \cong \rl$ in this example, it again remains to show the mapping is injective.

We can proceed akin to Example \ref{ex man}. Indeed, if two elements $v$ and $w$ of $F$ whose only nontrivial terms were contained in $T_0(\varphi(U))$, now corresponding to plot $\pi \circ \varphi^{-1}$, were equivalent, then we would have $v-w\in R$, meaning

$$(\pi \circ \varphi^{-1}, v - w) = \sum_ic_i[(\ol{p_i},v_i) - (\ol{q_i},g^i_*v_i)]$$

\noindent where $c_i\in \rl$, $\ol{p_i}, \ol{q_i}$ are pointed plots mapping $0$ to $[x]$, each $g^i$ is a germ of smooth maps such that $\ol{p_i}=\ol{q_i}\circ g^i$ as germs, and $v_i\in T_0(U_{\ol{p_i}})$. The above equality holds in $F$. One can then expand and group our plots as in Example \ref{ex man}, obtaining

\begin{align*}
    (\pi \circ \varphi^{-1}, v - w) &= \sum_ic_i[(\pi \circ \varphi^{-1},(\varphi \circ p_{i})_*v_i)-(\pi \circ \varphi^{-1},(\varphi \circ q_{i})_*g^i_*v_i)] \\
    &+ \sum_{j=1}^{m_1}c^1_j[(\ol{p_1},v^1_j)-(\pi\circ \varphi^{-1},(\varphi \circ p_{1})_*v^1_j)] \\
    & \quad \quad +\sum_{k=1}^{n_1} c^1_k[(\pi \circ \varphi^{-1},(\varphi \circ p_1)_*(g^{1,k})_*v^1_k)-(\ol{p_1},(g^{1,k})_*v^1_k)] \\
    & \vdots \\
    &+ \sum_{j=1}^{m_N} c^N_j[(\ol{p_N},v^N_j)-(\pi \circ \varphi^{-1},(\varphi \circ p_{N})_*v^N_j)] \\
    & \quad \quad +\sum_{k=1}^{n_N} c^N_k[(\pi \circ \varphi^{-1},(\varphi \circ p_{N})_*(g^{N,k})_*v^N_k)-(\ol{p_N},(g^{N,k})_*v^N_k)]
\end{align*}

Now, because we have $\ol{p_i} = \ol{q_i}\circ g^i$ as germs, locally about $0$ it must be the case for $w\in U_{\ol{p}}$ that $p_i(w)=\pm q_i(g^i(w))$ (recall from the discussion surrounding \ref{man well-defined} that the $p_i$ and $q_i$ plots such that $\ol{p_i} = \pi \circ p_i$ and $\ol{q_i}=\pi \circ q_i$ can be chosen unambiguously). Then, because $p_i(0)=q_i(g^i(0))=|x|>0$, we must have $p_i=q_i\circ g^i$ as germs as well. Thus, $\varphi_*p_{i*}=\varphi_*q_{i*}g^i_*$. With this dealt with, one can proceed analogously to Example \ref{ex man} to conclude that $T_{[x]}(\rl/O(1))=F/R \cong T_0(\varphi(U))\cong \rl$ when $[x]\neq 0$.

We now show $T_{[0]}(\rl / O(1))=\{0\}$. For any internal tangent vector of the form $(\ol{p},v)\in F/R=T_{[0]}(\rl / O(1))$, because $\ol{p}$ lies in the quotient diffeology on $\rl / O(1)$ it is the case that locally $\ol{p}=\pi \circ p$ for $p$ a plot in the standard diffeology on $\rl$. Because $\pi$ is a pointed plot when $[x]=[0]$, we have the following commutative diagram:

$$
\xymatrix{
U_{\ol{p}} \ar[rr]^{p} \ar[dr]_{\ol{p}} & & \rl \ar[dl]^{\pi} \\
& \rl/O(1) &
}
$$

That is, $(\ol{p},v)-(\pi,p_*(v))$ is in $R$. So, to complete our argument we need to show $(\pi,v)=0$ for all $v\in T_0(\rl)$. If we take $g:\rl \ra \rl$ to be the automorphism $g(x)=-x$ (that is, the action of $-1\in O(1)$ on $\rl$), we have the following diagram:

$$
\xymatrix{
\rl \ar[rr]^{g} \ar[dr]_{\pi} & & \rl \ar[dl]^{\pi} \\
& \rl/O(1) &
}
$$

\noindent where we have $\pi = \pi \circ g$ as germs (in fact, they are identically equal), meaning, for any $v\in T_0(\rl)$, $(\pi,v) \sim (\pi, g_*(v))=(\pi,-v)$. That is, in $T_{[0]}(\rl / O(1))=F/R$, we have $0=(\pi,v) + (\pi,-v)=(\pi,2v)$, for any $v\in T_0(\rl)$. This is the desired result.
\end{ex}

While illustrative, there are still several factors at play in Example \ref{ex r1} which differentiate it from the general case. In particular, the action of $O(1)$ is a compact linear group action on a Euclidean space consisting of finitely-many well-understood elements. The general problem requires a more sophisticated tack.

We close this chapter with two theoretical results concerning the internal tangent space to a point of a diffeological space. First, it is known that the relations between internal tangent vectors are determined by the two-dimensional plots.

\begin{prop} [\cite{tan_space_and_bundle} Proposition 3.4] \label{prop 3.4}
    Given a diffeological space $X$, let $X'$ denote the diffeological space with the same underlying set but with diffeology generated by all plots $p:\rl^2 \ra X$, for $p\in \dc_X$. The spaces $T_x(X')$ and $T_x(X)$ are isomorphic.
\end{prop}

\begin{ex}
    This example, drawn from \cite{tan_space_and_bundle} Example $3.22$, shows that we indeed need two-dimensional plots in Proposition \ref{prop 3.4}. That is, the internal tangent space of a diffeological space at a point is NOT determined by the one-dimensional plots. Let $X=\rl^n$ be equipped with the wire diffeology introduced in Example \ref{ex wire}. It is the case that $T_x(X)$ has uncountable dimension when $n\ge 2$ for any point $x\in X$. In particular, the higher-dimensional plots which allow for dimension-reducing relations do not in general factor through the $1$-dimensional plots, and so aren't present. For example, in the case where $n=2$, the internal tangent vectors $(p_{\alpha},\frac{d}{dt})$ for $\alpha \in \rl$ are all linearly independent, where $p_{\alpha}:\rl \ra X$ sends $x$ to $(x,\alpha x)$.
\end{ex}

Another known result is that, as with the tangent space to a point of a manifold, the notion of internal tangent space to a point of a diffeological space is ``local." 

\begin{lem}[\cite{tan_space_and_bundle} Proposition 3.6]\label{lem local}
    Let $A$ be a $D$-open neighborhood of $x$ in a diffeological space $X$, equipped with the subset diffeology. The spaces $T_x(A)$ and $T_x(X)$ are isomorphic.
\end{lem}

\begin{rem}
    In light of the discussion preceding Example \ref{ex d-top}, when $X$ is a manifold or a quotient space arising from a manifold, the condition of being $D$-open can be replaced by being open in the standard topology or open in the quotient topology, respectively, in Lemma \ref{lem local}.
\end{rem}

\chapter[Lie Theory]{\textbf{Chapter III. Lie Theory}}\label{sec_lie_theory}

In order to understand the general structure of the internal tangent space to a point of an orbit space of a manifold acted upon by a proper Lie group action, we need to first introduce the necessary definitions and results concerning Lie groups and their actions, along with those concerning Lie algebras. Our primary reference will be \cite{dk_lie}, though material from \cite{lee_smooth}, \cite{tu_man}, and \cite{strat} is also helpful. In this chapter, diffeomorphism means diffeomorphism of manifolds in the usual sense, unless explicitly noted.

\begin{defn}
    A \textit{Lie group} is a group $G$ which also has the structure of a smooth manifold, under which the group operations of multiplication and inversion, given below, are smooth.
    $$\text{multiplication} : G \times G \ra G, \ \  (a,b)\mapsto ab$$
    $$\text{inversion} : G \ra G, \ \ a\mapsto a^{-1} $$
    
    We shall denote the identity element of a Lie group by $1$.
\end{defn}

\begin{defn}
    The tangent space to the identity, $\gf=T_1G$, provided with the usual Lie bracket, is called the \textit{Lie algebra} of the Lie group $G$.
\end{defn}

Our focus is on the orbit spaces of manifolds under Lie group actions, in particular when these orbit spaces are not manifolds. In this paper, we define an action as follows.

\begin{defn} \label{action}
    A \textit{Lie group action} (or herein just ``action") of a Lie group $G$ on a manifold $M$ is a smooth mapping $A:G\times M \ra M$ such that
    $$A(gh,x)=A(g,A(h,x)), \ \ g,h\in G, x\in M.$$   
    We will also often suppress the action and write $g\cdot x$ for $A(g,x)$, and we will also sometimes write $A(g,x)$ as either $(A(g))(x)$ or $A(g)(x)$.
\end{defn}

\begin{rem}
    Throughout this document, ``action" will mean precisely the above. Therefore, all actions considered arise from Lie groups, are smooth, and act on manifolds as described, even if not explicitly mentioned. The above is more specifically a ``left action," and right actions are defined analogously.
\end{rem}

An action of a group $G$ on a manifold $M$ is said to be \textit{proper} if the mapping

$$(g,x) \mapsto (g\cdot x, x)$$

\noindent is a proper mapping $G\times M \ra M\times M$, meaning that the inverse image of every compact set in the codomain is compact in the domain.

For each $x\in M$, for $M$ a manifold, the \textit{orbit} of the action $A$ through $x$ is the set:

$$A(G)(x)=G\cdot x=\{g\cdot x \ | \ g\in G\}$$

\begin{rem}
    We can also think of the orbit through $x$ as the image of the mapping $A_x : G \ra M$ such that $g \mapsto A(g)(x)$. As this will be a useful notion later, we note it here.
\end{rem}

We denote by $M / G$ the collection of orbits (which are equivalence classes, because the set of orbits forms a partition for $M$) of an action of $G$ on a manifold $M$ and by $\pi : M \ra M/G$ the canonical projection. Some authors use the notation $G \backslash M$ for the orbits of a left action, reserving $M/G$ for the orbits of a right action, but we will not follow this convention. Assign the quotient topology and quotient diffeology to $M / G$ (via $\pi$), and note that this topology is Hausdorff when the action is proper (\cite{dk_lie} Lemma 1.11.3).

An action is said to be \textit{transitive} if $G\cdot x = M$ for some $x\in M$. For each $x\in M$, we define the \textit{stabilizer} of $x$ under the action to be the subgroup of $G$ given by:

$$G_x=\{g\in G \ | \ g\cdot x = x\}$$

The stabilizer can be shown to be compact when the action is proper. Indeed, for a proper action of a Lie group $G$ on $M$ consider the inverse image of the compact set $\{(x,x)\}\in M\times M$ under the mapping $(g,x) \mapsto (g\cdot x, x)$.

An action is \textit{free} if $G_x=\{1\}$ for all $x\in M$. It is shown in \cite{dk_lie} (Theorem $1.11.4$) that when $G$ admits a proper and free action on a manifold $M$ there is a unique manifold structure on $M / G$. Therefore, for our purposes in determining the internal tangent space to points of quotients of manifolds by proper Lie group actions, we are interested in actions which are proper but not free. We now address several more specialized results and definitions.

\begin{defn}\label{equivalence}
    If $A$, and $B$, are actions of a Lie group $G$ on manifolds $X$, and $Y$, respectively, we say a mapping $\Phi : X \ra Y$ is \textit{$G$-equivariant} if
    $$(\Phi \circ A(g)) (x) = (B(g) \circ \Phi)(x), \ \ g\in G$$
    \noindent for all $x\in X$. If the map $\Phi$ is a $G$-equivariant diffeomorphism then we call it an \textit{equivalence} of actions, and say the actions are \textit{equivalent}.
\end{defn}

In light of this, note that the below theorem says that the action of a \textit{compact} Lie group $K$ on a manifold $M$ which fixes a point $x_0$ is, when restricted to a specified $K$-invariant open neighborhood $U$ of $x_0$, equivalent to a linear action of $K$ on $T_{x_0}M$, suitably restricted to an open neighborhood of $0$ (see \cite{dk_lie}, p. 98).

\begin{thm}[\cite{dk_lie} Theorem 2.2.1; Bochner's Linearization Theorem] \label{bochner}
    Let $A$ be a Lie group action of a compact group $K$ on $M$ and let $x_0\in M$ be such that $A(k)(x_0)=x_0$, for all $k\in K$. Then there exists a $K$-invariant open neighborhood $U\subseteq M$ of $x_0$ and a diffeomorphism $\chi$ from $U$ onto an open neighborhood $V\subseteq T_{x_0}M$ of $0$, such that:
    $$\chi(x_0) = 0, \ \ \chi_*|_{x_0}=id : T_{x_0}M \ra T_{x_0}M$$
    \noindent and:  
    $$\chi(A(k)(x)) = (A(k)_*|_{x_0})(\chi(x)), \ \ k\in K, x\in U$$
\end{thm}

Note that above and throughout this document, we identify $T_0(T_xM)$ with $T_xM$, hence the codomain of our differential map in Theorem \ref{bochner}.

\begin{rem}
    Within the scope of this document, the \textit{tangent action} is defined as follows: given an action--denoted by $A$--of $K$ a compact Lie group on a manifold $M$ which fixes a point $x\in M$, the tangent action by an element $k\in K$ on $T_xM$ is the differential of the map $A(k):M\ra M$ mapping $x$ to $k\cdot x$, at $x$. That is, the tangent action, itself a compact action, can be thought of as the $\ci$ homomorphism sending $k\mapsto (A(k))_*|_{x}$.
\end{rem}

In our work, $K$ will usually be the stabilizer to a point $x_0$, namely $G_{x_0}$, in a given manifold $M$ acted on by a Lie group $G$, which we recall is compact when the action of $G$ on $M$ is proper. 

\begin{defn}[\cite{dk_lie} Definition 2.3.1]\label{slice def}
    Let $A:G\times M \ra M$ be an action of a Lie group $G$ on a manifold $M$. A (smooth) \textit{slice} at $x_0\in M$ for the action $A$ is a submanifold $S$ of $M$ through $x_0$ such that:
    \begin{enumerate}
        \item $T_{x_0}M = (A_{x_0})_*|_{1}(\gf) \oplus T_{x_0}S;$ and $T_xM=(A_x)_*|_{1}(\gf) + T_xS, \ x\in S$;
        \item $S$ is $G_{x_0}$-invariant;
        \item if $g\in G, x\in S$, and $g\cdot x\in S$, then $g\in G_{x_0}$.
    \end{enumerate}
\end{defn}

Recall that $A_{x_0}:G\ra M$ is the map sending $g$ to $A(g)(x_0)$.

\begin{rem}
    $G\cdot x_0$ is an immersed submanifold of $M$, and we have $(A_{x_0})_*|_{1}(\gf)=T_{x_0}(G\cdot x_0)$, as mentioned in \cite{dk_lie} on pages $56-57$ and $95$. 
\end{rem}

Slices are a crucial tool available when studying a proper Lie group action.

\begin{thm}[\cite{dk_lie} Theorem 2.3.3; Slice Theorem] \label{slice}
    For $A$ a proper action of a Lie group $G$ on a manifold $M$, there exists a slice $S$ through $x_0$ (which depends on $x_0$), for any $x_0\in M$.
\end{thm}

The following theorem draws on the results above to show that, in effect, we can locally about a point $x_0$ in $M$ view any proper action of a Lie group $G$ on said manifold $M$ as the compact, linear tangent action of $G_{x_0}$ on $T_{x_0}S$, thereby simplifying our study of the internal tangent space (see chapter \ref{sec_its_structure}). Before stating this result, we explain the central notation and actions involved.

Let $X$ and $Y$ be smooth manifolds and $H$ a Lie group action acting on both $X$ and $Y$ (the actions may be distinct). Assume the action of $h\in H$ on $X$ is proper and free, and denote it by $x\mapsto x\cdot h^{-1}$. Denote the action of $H$ on $Y$ by $y\mapsto h\cdot y$. Due to Theorem 1.11.4 of \cite{dk_lie}, the orbit space $X/H$ is a manifold. The action of $H$ of $X \times Y$, defined by:

$$(h,(x,y))\mapsto (x\cdot h^{-1}, h\cdot y), \ \ \ h\in H, (x,y)\in X\times Y$$

\noindent is proper and free as well (see \cite{dk_lie} section $2.4$). Therefore, taking $[x,y]$ to be the orbit of $(x,y)$ in $X\times Y$ under the action of $H$, the quotient space is a smooth manifold, denoted:

$$X\times_H Y = \{[x,y] \ | \ (x,y)\in X\times Y\}$$

Now let $G$ be another Lie group which acts on $X$, such that furthermore its action \textit{commutes} with the action of $H$ on $X$ given above. In this case, the action $(g,(x,y)) \mapsto (g\cdot x,y)$ of $G$ on $X\times Y$ also commutes with the action of $H$ on $X\times Y$, so its action on $X\times_H Y$ is well-defined (again, see \cite{dk_lie} section $2.4$).

In Theorem \ref{tube}, below, we consider the specific case when $H$ is a closed Lie subgroup of $G$, $X=G$, $H$ acts on $G$ via right multiplication, and $G$ acts on $G$ via left multiplication. Note that $G$ acts transitively on $G/H$ in this case.

\begin{thm}[\cite{dk_lie} Theorem 2.4.1; Equivariant Tube Theorem] \label{tube}
    Let $A$ be a proper action of a Lie group $G$ on a manifold $M$. About any point $x_0\in M$, there exists a $G$-invariant open neighborhood $U$ of $x_0$ such that the $G$-action on $U$ is equivalent to the action of $G$ on $G\times_{G_{x_0}}B$. Here, $B$ is an open $G_{x_0}$-invariant neighborhood of $0$ in $T_{x_0}M / ((A_{x_0})_*|_{1})(\gf)$, on which the compact group $G_{x_0}$ acts linearly, via the tangent action of $(A(k))_*|_{x_0}$ (for $k\in G_{x_0}$) modulo $((A_{x_0})_*|_{1})(\gf)$.
\end{thm}

Recall that it can be shown that $(A_{x_0})_*|_{1}(\gf)=T_{x_0}(G\cdot x_0)$. We utilize this going forward. There are some important structural aspects of the tangent space $T_{x_0}M$ of a manifold $M$ at a point $x_0$ and the action on it by the tangent action of the compact stabilizer $G_{x_0}$. We address below the salient points, and then indicate how they relate to slices, Theorem \ref{slice}, and Theorem \ref{tube}. Throughout the below material, we consider a point $x_0$ of a manifold $M$ of dimension $n$ acted upon by a proper Lie group action of a group $G$ with stabilizer at $x_0$ denoted $G_{x_0}$. The action of $G_{x_0}$ is the relevant tangent action until noted. 

As discussed in \cite{dk_lie}, in particular pages $96-97$, from an arbitrary inner product on $T_{x_0}M$ one can always construct an inner product which is invariant under the tangent action of $G_{x_0}$.

It is shown in the proof of Theorem \ref{slice} from \cite{dk_lie} that the subspace $T_{x_0}(G\cdot x_0)$ is invariant under the tangent action by $G_{x_0}$, hence it is an easy exercise to show that its orthogonal complement $T_{x_0}(G\cdot x_0)^{\perp}$ is as well. This uses the invariant inner product guaranteed above.

Further, it can be shown in this case--where $G_x$ is compact and acts linearly on $T_xM$ with given invariant inner product--that $T_xM$ can be identified with Euclidean space $\rl^n$ with its Euclidean inner product on which $G_x$ acts orthogonally. Therefore, we shall identify $T_xM$ and its $G_x$-invariant subspaces $T_{x_0}(G\cdot x_0)$ and $T_{x_0}(G\cdot x_0)^{\perp}$ with the Euclidean space $\rl^k$ (for the appropriate $k$ in each case) in the work below.

It also follows from the proof of Theorem \ref{slice} given in \cite{dk_lie} that the slice $S$ to a point $x_0\in M$ corresponding to the proper action of a Lie group is $G_{x_0}$-equivariantly diffeomorphic to $T_{x_0}(G\cdot x_0)^{\perp}$ (here the action of $G_{x_0}$ on $S$ is that of a subgroup of $G$, and on $T_{x_0}(G\cdot x_0)^{\perp}$ is the tangent action on the invariant subspace). Indeed, directly from the proof we have that our slice $S$ is $G_{x_0}$-equivariantly diffeomorphic to $V\cap T_{x_0}(G\cdot x_0)^{\perp}$ where $V$ is a ball of some fixed radius $\epsilon$ about $0$ in $\rl^n$ and $T_{x_0}(G\cdot x_0)^{\perp} = \rl^k$ is a subspace of $\rl^n$. There is, however, a natural $G_{x_0}$-equivariant diffeomorphism from $V \cap T_{x_0}(G\cdot x_0)^{\perp}$  to $\rl^k = T_{x_0}(G\cdot x_0)^{\perp}$ given by

$$\kappa(x)=\frac{1}{\epsilon^2-||x||^2}x$$

\noindent in light of the fact that $||\cdot ||$, the norm induced from the invariant inner product above, is invariant under the action.

Furthermore, $T_{x_0}(G\cdot x_0)^{\perp}$ can be $G_{x_0}$-equivariantly identified with $T_{x_0}M/T_{x_0}(G\cdot x_0)$. Indeed, given an element $v^{\perp}\in T_{x_0}(G\cdot x_0)^{\perp}$, the mapping $\omega$ defined by $\omega(v^{\perp})=[v^{\perp}]\in T_{x_0}M/T_{x_0}(G\cdot x_0)$ is a $G_{x_0}$-invariant mapping due to the invariance of $T_{x_0}(G\cdot x_0)$ and $T_{x_0}(G\cdot x_0)^{\perp}$ under the tangent action. When the $G_{x_0}$-action on $T_{x_0}M$ has descended to the quotient, as here where $h\cdot[v^{\perp}]:=[h\cdot v^{\perp}]$ (for $h\in G_{x_0}$), we sometimes refer to it as the \textit{isotropy action}. We record the conclusion of this discussion below.

\begin{prop}\label{slice to isotropy}
    Let $x_0$ be a point of a manifold $M$ of dimension $n$ acted upon by a proper Lie group action of a group $G$ with stabilizer at $x_0$ denoted $G_{x_0}$ and with slice $S$ through $x_0$ as in Theorem \ref{slice}. In the language of Definition \ref{equivalence}, $S$, $T_{x_0}(G\cdot x_0)^{\perp}$, and $T_{x_0}M/T_{x_0}(G\cdot x_0)$ are $G_{x_0}$-equivalent. The action on $S$ is the action of $G_{x_0}$ thought of as a subgroup of $G$, the action on $T_{x_0}(G\cdot x_0)^{\perp}$ (thought of as an invariant subspace of $T_{x_0}M$) is the tangent action associated to $G_{x_0}$, and the action on $T_{x_0}M/T_{x_0}(G\cdot x_0)$ is the isotropy action resulting from the tangent action. 
\end{prop}

One of the implications of this as related to Theorem \ref{tube} and important in the following chapter is that for our purposes we can and shall take $B$, the open $G_{x_0}$-invariant neighborhood of zero in $T_{x_0}M/T_{x_0}(G\cdot x_0)$ described in said theorem, to be $T_{x_0}M/T_{x_0}(G\cdot x_0)$ itself. This is immediate from the proof in \cite{dk_lie}, page $103$, and the above paragraphs.

\chapter[The Structure of the Internal Tangent Space]{\textbf{Chapter IV. The Structure of the Internal Tangent Space}}\label{sec_its_structure}

We are now in a position to describe the general internal tangent space to a point of a quotient of a manifold by a proper Lie group action. The key will be the ability to utilize (especially) Theorem \ref{tube} to reduce this problem to the case of a compact group acting linearly on a suitable space.

Indeed, suppose we are given a manifold $M$ and a proper Lie group action $A$ of a Lie group $G$ on $M$. We wish to determine $T_{[x]}(M / G)$ for an arbitrary $[x]\in M/G$. Note first that, in this case, as vector spaces we have dim$(T_{[x]}(M/G)) \le \text{dim}(T_xM)$.

The key simplification of the actual problem is as follows. By Theorem \ref{tube}, we have a $G$-invariant open neighborhood $U$ of $x_0$ in $M$ such that the $G$-action on $U$ is equivalent to the action of $G$ on $G\times_{G_{x_0}}B$, with $B=T_{x_0}M/T_{x_0}(G\cdot x_0)$. In particular, there is a $G$-equivariant diffeomorphism $\Phi:U \ra G\times_{G_{x_0}}B$. The following lemma will turn this consequence into Corollary \ref{tube quotient}. Note that $\Phi$ is both a diffeomorphism in the sense of smooth manifolds and, immediately, a diffeological diffeomorphism. In this chapter we return to the overall convention that ``diffeomorphism" refers to a diffeological diffeomorphism unless noted otherwise. We begin with the following lemma.

\begin{lem}\label{lem descend}
    Let $X$ and $Y$ be manifolds with their standard diffeologies both acted upon by Lie group actions of a group $G$, and $\phi:X\ra Y$ a $G$-equivariant diffeomorphism between them. Then $\phi$ descends to a $G$-equivariant diffeomorphism on the quotients $\widetilde{\phi} : X/G \ra Y/G$.
\end{lem}

\begin{proof}
    Because $\phi$ is $G$-equivariant, it maps $G$-orbits onto $G$-orbits and so descends to a map between the quotients, denoted $\widetilde{\phi}$, where $\widetilde{\phi}([x])=[\phi(x)]$, for $x\in X$. It remains to show this map is a diffeomorphism.
    
    First, we show the map is smooth. Given a plot $\ol{p}(t)$ in the quotient diffeology of $X/G$, we claim that $\widetilde{\phi}(\ol{p}(t))$ is a plot in the quotient diffeology of $Y/G$. Indeed, locally about any point $t_0$ in its domain, we have $\ol{p}(t)=\pi(p(t))=[p(t)]$ for $p$ a plot in the diffeology on $X$ and $\pi$ the canonical projection. Thus, for $t$ in a neighborhood of $t_0$, 
    
    $$\widetilde{\phi}(\ol{p}(t))=\widetilde{\phi}([p(t)])=[\phi(p(t))]$$

    If $p'$ is another plot in the diffeology on $X$ such that $\ol{p}(t)=\pi(p'(t))=[p'(t)]$ near $t_0$, we also have for $t$ in a neighborhood of $t_0$

    $$\widetilde{\phi}(\ol{p}(t))=\widetilde{\phi}([p'(t)])=[\phi(p'(t))]$$

    However, because $\ol{p}(t)=[p(t)]=[p'(t)]$ as germs about $t_0$ and $\phi$ is $G$-equivariant, we have that $[\phi(p(t))]=[\phi(p'(t))]$ in another, perhaps smaller neighborhood of $t_0$. Hence, the map $\widetilde{\phi}(\ol{p}(t))$ has an unambiguous meaning locally. Because $\phi(p(t))$ is a plot in $Y$ due to the smoothness of $\phi$, meaning $\pi(\phi(p(t)))=[\phi(p(t))]$ is a plot in the quotient diffeology for $Y/G$, and because $\widetilde{\phi}(\ol{p}(t))=[\phi(p(t))]$ in a neighborhood of $t_0$, by the locality condition of diffeologies (item $3$ in Definition \ref{diff def}) $\widetilde{\phi}(\ol{p}(t))$ is globally a plot in the quotient diffeology of $Y/G$ as well (since $t_0$ was an arbitrary domain point). Thus, $\widetilde{\phi}$ maps plots to plots and so is smooth.

    Further, since the map $\phi$ is an equivariant diffeomorphism, its inverse $\phi^{-1}$ also descends to the quotients, now denoted $\widetilde{\phi}^{-1}$, sending $\widetilde{\phi}^{-1}([y])=[\phi^{-1}(y)]$ for $y\in Y$. This map is likewise smooth and is still an inverse map, as

    $$\widetilde{\phi}^{-1}(\widetilde{\phi}([x]))=[\phi^{-1}(\phi(x))]=[x]$$
    $$\widetilde{\phi}(\widetilde{\phi}^{-1}([y]))=[\phi(\phi^{-1}(y))]=[y]$$

    Thus $\widetilde{\phi}$ is a diffeomorphism $X/G \ra Y/G$.
\end{proof}

This leads immediately to the following.

\begin{cor}\label{tube quotient}
    The $G$-equivarant diffeomorphism $\Phi: U \ra G\times_{G_{x_0}}B$ given by Theorem \ref{tube} descends to a $G$-equivariant diffeomorphism on the quotients:
    $$\widetilde{\Phi} : U/G\ra (G\times_{G_{x_0}}B)/G$$
\end{cor}

Now, we link the orbit space $M/G$ which is the focus of our study to a space which has an oftentimes nicer collection of internal tangent spaces.

\begin{thm}\label{local linear compact}
    Let $M$ be a manifold and $G$ a Lie group acting properly on $M$. Given a point $x_0\in M$, with stabilizer $G_{x_0}$ under the action, the following spaces are diffeomorphic, where $\widetilde{U} \subseteq M/G$ is the neighborhood $U/G$ of $[x_0]\in M/G$ coming from Corollary \ref{tube quotient}, $G\times_{G_{x_0}}B$ is as described in and preceding Theorem \ref{tube}, $B=T_{x_0}M/T_{x_0}(G\cdot x_0)$ as described following Theorem \ref{tube}, and $S$ is a slice through $x_0$ in M guaranteed by Theorem \ref{slice}.
    $$M/G \supseteq \widetilde{U} \cong (G\times_{G_{x_0}}B)/G \cong B/G_{x_0} \cong T_{x_0}S/G_{x_0}$$
    \noindent where the action on $G\times_{G_{x_0}}B$ is multiplication of the first component by $G$, the action on $B$ is the isotropy action induced by the tangent action of $G_{x_0}$, and the action on $T_{x_0}S=T_0(T_{x_0}(G\cdot x_0)^{\perp})=T_{x_0}(G\cdot x_0)^{\perp}$ is the tangent action of $G_{x_0}$.
\end{thm}

\begin{proof}
    The first isomorphism, $\widetilde{U} \cong (G\times_{G_{x_0}}B)/G$, is Corollary \ref{tube quotient}. For the second isomorphism, observe that $(G\times_{G_{x_0}}B)/G={[1]\times_{G_{x_0}}B}$, which is immediately diffeomorphic to $B/G_{x_0}$ due to the definition of the manifold $G\times_{G_{x_0}}B$ (the second components of plots in the standard diffeology on this manifold are the plots of $B/G_{x_0}$). This step is also found in both \cite{watts_thesis}, Lemma $3.17$, and \cite{watts_paper}, Lemma $5.9$. The last isomorphism follows from Proposition \ref{slice to isotropy} and Lemma \ref{lem descend}, after considering the resulting differentials. This is due to the fact that $T_{x_0}S=T_0(T_{x_0}(G\cdot x_0)^{\perp})$, and that $T_0(T_{x_0}(G\cdot x_0)^{\perp})=T_{x_0}(G\cdot x_0)^{\perp}$.
    
    Note that $[x_0]\in M/G$ is mapped to $[0]$ in $B/G_{x_0}$ and $T_{x_0}S/G_{x_0}$. This follows from Theorem \ref{bochner}, which is used in the proof of Theorem \ref{slice} and, through this, Theorem \ref{tube}.
\end{proof}

The below theorem is the main result of this chapter, though in our work we will employ a corollary of it which utilizes Theorem \ref{local linear compact}.

\begin{thm}\label{main result}
    The internal tangent space to a point is a diffeological invariant. That is, if $\chi : X \ra Y$ is a diffeomorphism mapping a point $x$ to $y=\chi(x)$, then $T_x(X) \cong T_y(Y)$.
\end{thm}

\begin{proof}
    Recalling Definition \ref{def 3.1}, suppose $T_x(X) = F/R$ and $T_y(Y)=F'/R'$, where $F=\oplus_pT_0(U_p)$, $F'=\oplus_{p'}T_0(U_{p'})$, and $R$ and $R'$ consist of the relations spanned by vectors of the form $(p,v)-(q,g_*(v))$ and $(p',v')-(q',g_*(v'))$, respectively.

    As $X$ is diffeomorphic to $Y$, each plot $p\in 
    \dc_X$ the diffeology of $X$ corresponds to a plot $\chi(p)\in \dc_Y$ the diffeology of $Y$ with domain $U_p=U_{\chi(p)}$, and conversely. Therefore, $F=F'$. We now show the basic relations in $R$ correspond exactly with those of $R'$. Indeed, for $p,q\in \dc_X$ forming a basic relation, and their corresponding plots $p'=\chi(p),q'=\chi(q)\in\dc_Y$, we have

    \begin{align*}
        (p,v) - (q,g_*v) \in R &\iff p=q\circ g \\
        &\iff \chi^{-1}\circ p' = \chi^{-1} \circ q' \circ g \\
        &\iff \chi \circ \chi^{-1}\circ p' = \chi \circ \chi^{-1} \circ q' \circ g \\
        &\iff p' = q' \circ g \\
        &\iff (p',v) - (q',g_*v) \in R'
    \end{align*}

    \noindent which tells us a difference of two elements forms a basic relation in $R$ if and only if their corresponding elements form a basic relation in $R'$; thus $R\cong R'$. All equalities on the right hand side are equalities of germs. Thus, $F/R \cong F'/R'$, yielding the desired result.
\end{proof}

The following is an immediate consequence of the above Theorem \ref{main result}, Theorem \ref{local linear compact}, and the locality of the internal tangent space spelled out in Lemma \ref{lem local}.

\begin{cor}\label{main result cor}
    Let $M$ be a manifold acted upon by a proper Lie group $G$ with notation as in Theorem $\ref{local linear compact}$. Then the problem of determining the internal tangent space of such a quotient reduces to the study of a suitable space quotiented by a compact, linear group action. In detail:
    $$T_{[x_0]}(M/G) \cong T_{[0]}(T_{x_0}S/G_{x_0})$$
\end{cor}

\begin{rem}
    In light of Theorem \ref{local linear compact} and Proposition \ref{slice to isotropy}, we also have:
    
    $$T_{[x_0]}(M/G) \cong T_{[0]}(B/G_{x_0}) \cong T_{[0]}(T_{x_0}(G\cdot x_0)^{\perp}/G_{x_0})$$

\end{rem}

We have not identified $T_{[0]}(T_{x_0}S/G_{x_0})$ with $T_{x_0}S/G_{x_0}$, because $T_{x_0}S/G_{x_0}$ is not guaranteed to be a manifold, unlike in prior cases where we have made the analogous identification. Consider now the task as we have reduced it: that of determining $T_{[0]}(T_{x_0}S/G_{x_0})$. In general, we cannot further simplify the problem, but reducing the general case to that of a compact, linear Lie group action is significant.

\begin{rem}
    For instance, in the case $\rl / O(1)$ studied above, our slice for the action at $0$ is simply $\rl$ and $G_0 = O(1)$, so our findings do not make our earlier work obsolete.
    
    In general, there is ``elbow grease" required to precisely determine the internal tangent space to a point of the orbit space of a manifold acted upon by a proper Lie group action, but as the below material helps to show, our findings proactively make the process more manageable and help to link our work to the study of stratified spaces.
\end{rem}

\begin{ex}\label{ex r3}
We now work through an example where the results of this chapter are gainfully used. We will determine the internal tangent space to an arbitrary point $[x]\in \rl^3 / SO(3)$ (equipped with the quotient diffeology), where $SO(3)$ is the special orthogonal group of $3\times 3$ matrices acting on $\rl^3$ by multiplication (i.e., rotation). Although this result is not novel, it exhibits the utility of our work so far. We consider two cases.

When $[x]=[0]$, we show $T_{[0]}(\rl^3/SO(3))=\{0\}$. As in the determination of $T_{[0]}(\rl/O(1))$ in Example \ref{ex r1}, it is sufficient to show $(\pi,v)=0$ for all $v\in T_0(\rl^3)$, as $\rl^3$ is the domain of $\pi : \rl^3 \ra \rl^3 / SO(3)$. Consider the three elements $g_1, g_2, g_3 \in SO(3)$ corresponding to the matrices

$$
\begin{pmatrix}
    -1 & 0 & 0 \\
    0 & 1 & 0 \\
    0 & 0 & -1
\end{pmatrix}, \ \
\begin{pmatrix}
    -1 & 0 & 0 \\
    0 & -1 & 0 \\
    0 & 0 & 1
\end{pmatrix}, \ \
\begin{pmatrix}
    1 & 0 & 0 \\
    0 & -1 & 0 \\
    0 & 0 & -1
\end{pmatrix}
$$

Then, for an arbitrary $v\in T_0(\rl^3)$ we have the equivalences $(\pi, v) \sim (\pi, (g_i)_*(v)), i \in \{1,2,3\}$ in the internal tangent space (viewed as $F/R$, as in Definition \ref{def 3.1}). Thus, as the differentials $(g_i)_*$ are just matrix multiplication by $g_i$ in this case, we can directly compute (in the internal tangent space)

\begin{align*}
    (\pi,4v) &= (\pi, v) + (\pi, (g_1)_*(v)) + (\pi, (g_2)_*(v)) + (\pi, (g_3)_*(v)) \\
    &= (\pi, v) + (\pi, g_1\cdot v) + (\pi, g_2\cdot v) + (\pi, g_3\cdot v) \\
    &= 0
\end{align*}

\noindent hence $(\pi,v)=0$ for any $v\in T_0(\rl^3)$. This implies the desired result.

When $[x]\ne[0]$, we employ Theorem \ref{main result cor}. We claim $T_{[x]}(\rl^3/SO(3))=\rl$. As a slice for $x\in M$, we take an open interval on the unique line between $0$ and $x$ containing $x$ but not $0$. On this slice, the tangent action of $SO(3)_{x}$, the stabilizer of the action of $SO(3)$ at $x$ (which consists of all rotations with axis of rotation passing through $x$), is trivial.

Indeed, since the tangent action of the differentials of elements of the action of $SO(3)$ on $\rl^3$ can be identified with the action of the elements themselves in this instance, note that the tangent action of $SO(3)_{x}$ is trivial because the action of $SO(3)_{x}$ as a subgroup of $SO(3)$ must be trivial on the slice $S$. Thus, by Corollary \ref{main result cor} and Example \ref{ex man}

$$T_{[x]}(\rl^3/SO(3)) = T_{[0]}(T_{x}(S)/SO(3)_{x}) = T_{[0]}(\rl / SO(3)_x) = \rl$$

\noindent as desired.
\end{ex}

\chapter[Stratified Spaces]{\textbf{Chapter V. Stratified Spaces}}\label{sec_strat_stuff}

Loosely speaking, a stratified space is a space partitioned into ``pieces," each of which are manifolds, joined in a certain way (details will be given below). As each point in the space lies in a piece--a manifold--there is a well-understood notion of ``stratified tangent space" to any point of the stratified space. Indeed, the stratified tangent space to the point is just the tangent space to the point viewed as sitting within the manifold that is the piece containing said point.

The quotient space of a manifold acted on by a proper Lie group action admits a well-understood stratification, called the ``orbit-type stratification," meaning that in addition to viewing these orbit spaces as diffeological spaces with the smooth structure given by a diffeology (in particular the quotient diffeology arising from the standard diffeology on the manifold) one can also view them as stratified spaces with a Sikorski smooth structure (see \cite{strat} and \cite{orbit_strat_1}), though this notion of smooth structure goes beyond what we will discuss in this work.

It is our goal in this chapter to show that the internal tangent space to a point of the orbit space of a manifold quotiented by a proper Lie group action (with the orbit space viewed as a diffeological space) is isomorphic to the stratified tangent space to the point (with the orbit space viewed as a stratified space with the orbit-type stratification). The fact that these two extensions of the concept of tangent space to a point agree is of value because the orbit-type stratification for an orbit space is well-understood and so certain results attained in that context can then be passed on to the associated internal tangent space to a point.

The isomorphism between the internal tangent space to a point and the stratified tangent space to a point also assists with the comparison of the internal tangent bundle described in \cite{tan_space_and_bundle} and the stratified tangent bundle explained in, among other references, \cite{strat} and formed from the stratified tangent spaces to each point. However, such a comparison goes beyond the scope of this work.

\section[The Main Result]{\textbf{V.1. The Main Result}}

We now briefly introduce the requisite material on stratified spaces. One principal reference for this material is \cite{strat}, but we will also draw our exposition from \cite{dk_lie}. Recall that in a topological space $X$ a locally closed subset $A\subseteq X$ is a subset for which every point of $A$ has a neighborhood $U\subseteq X$ with $A\cap U$ closed in $U$.

\begin{defn}
    A \textit{stratified space} is a topological space equipped with a stratification. Let $X$ be a topological space. A \textit{stratification} of $X$ is a locally finite partition $\mathcal{S}$ satisfying the following two conditions:
    \begin{itemize}
        \item(Manifold Condition) Each piece of $\mathcal{S}$, called a stratum, is a locally closed smooth manifold.
        \item(Frontier Condition) If $S_1,S_2\in \mathcal{S}$ are such that $S_1\cap \ol{S_2} \neq \emptyset$, then $S_1$ is in the topological boundary of $S_2$.
    \end{itemize}
\end{defn}

The key points of note are, first, that a manifold $M$ is automatically a stratified space with a single stratum and, second, that a manifold $M$ acted upon by a proper Lie group action of a Lie group $G$ admits another, non-trivial stratification called the ``orbit-type stratification" (introduced below). Before proceeding, we must introduce several other definitions. For background, see \cite{dk_lie} sections $2.6\ \& \ 2.7$ and \cite{strat} chapter $4$.

Throughout, $M$ shall denote a manifold acted upon by a proper Lie group action of a Lie group $G$.

\begin{defn}[\cite{dk_lie} Definition $2.6.1$]
    We say that $x,y\in M$ are of the same \textit{orbit type}, denoted $x\sim y$, if there exists a $G$-equivariant bijection from $G\cdot x$ to $G\cdot y$.
\end{defn}

Note that $\sim$ is an equivalence relation on $M$. The equivalence classes are called \textit{orbit types} in $M$, and are denoted

$$M^{\sim}_x:=\{y\in M \ | \ y \sim x\}$$

\begin{defn}[\cite{dk_lie} Definition $2.6.5$, corrected by Jordan Watts]
    We say that $x, y\in M$ are of the same \textit{local type}, denoted $x \approx y$, if for any $G$-invariant open neighborhoods $U$ and $V$ of $x$ and $y$, respectively, after shrinking $U$ and $V$ if necessary, there exists a $G$-equivariant diffeomorphism from $U$ to $V$ sending $G\cdot x$ onto $G\cdot y$.

    We will also denote the above in terms of the orbits, written as $G\cdot x \approx G\cdot y$, and say that $$G\cdot x, \ G\cdot y \in M/G$$ are of the same local type.
\end{defn}

The relation $x \approx y$ is an equivalence relation, and each local action type is $G$-invariant and contained within an orbit type. We call the equivalence classes \textit{local action types}, denoted

$$M^{\approx}_x := \{y\in M \ | \ y\approx x\}$$

Likewise there is an equivalence relation on $M/G$, denoted

$$M^{\approx}_{G\cdot x}/G = \{G\cdot y \in M/G \ | \ G\cdot y \approx G\cdot x\}$$

We now have the following result, which shows that a manifold acted upon by a Lie group via a proper Lie group action admits another, non-trivial stratification: the ``orbit-type stratification." Recall the action by $G$ is assumed to be proper.

\begin{thm}[\cite{dk_lie} Theorem $2.7.4$ and \cite{strat} Theorem $4.3.7$]
    There is a (Whitney) stratification of $M$, called the \textit{orbit-type stratification}, which consists of the connected components of the orbit types in $M$.
\end{thm}

Although the above theorem is phrased in terms of orbit types, local action types will be of use to us below. Indeed, the significant fact about the orbit-type stratification is that it descends to the quotient $M/G$, and there we will employ the local action type.

First, we note the following lemma. For $H$ any closed subgroup of $G$, let

$$M_{(H)} := \{x\in M \ | \ \exists g\in G ~\text{s.t.}~gHg^{-1}=G_x\}$$

That is, $M_{(H)}$ denotes the set of points $x$ whose stabilizer $G_x$ is conjugate to $H$. The following is a consequence of Lemma $2.6.2(i)$ in \cite{dk_lie}.

\begin{lem}
    The orbit type of a point $x\in M$ is given by $M_{(G_x)}$, in the notation above.
\end{lem}

Now we are ready to investigate the stratification of $M/G$ resulting from the orbit-type stratification of $M$. Indeed, as a consequence of \cite{strat} Corollary $4.3.11$ the orbit-type stratification descends from $M$ to a stratification of $M/G$ (also called the orbit-type stratification). Immediately as a consequence of this corollary and the above lemma, or from \cite{dk_lie} Theorem $2.6.7$, we have that the strata, each of which is a manifold, are of the form $M^{\approx}_{G\cdot x}/G$. See \cite{dk_lie} Theorem $2.6.7(iii)$ for the proof that this space is indeed a manifold. Further, from \cite{dk_lie} Theorem $2.6.7(v)$ we have (recalling Definition \ref{slice def}, Theorem \ref{tube}, and the subsequent discussion):

$$\text{dim}(M^{\approx}_{G\cdot x}/G) = \text{dim}(T_xM / T_x(G\cdot x))^{G_x}$$

Here, the notation $(T_xM / T_x(G\cdot x))^{G_x}$ denotes the fixed point set of $T_xM / T_x(G\cdot x)$ under the isotropy action of $G_x$. In general, for $H$ a closed subgroup of $G$ a Lie group acting properly on M, the \textit{fixed point set} of a manifold $M$ with respect to $H$ is

$$M^H=\{y\in M \ | \ h\cdot y = y, \ \forall h\in H\}$$

As a consequence of Theorem \ref{slice} and Definition \ref{slice def}, the above simplifies to

$$\text{dim}(M^{\approx}_{G\cdot x}/G)=\text{dim}(T_xS)^{G_x}$$

\noindent where $S$ is a slice through $x$ in $M$, acted upon by the isotropy action of $G_x$ (or it can be viewed as the tangent action of $G_x$ on the invariant subspace $T_xS$ of $T_xM$). For a refresher, see the discussion following Theorem \ref{tube}.

It is not our purpose here to expound upon the nature of stratified spaces. Therefore, we have opted for the abridged treatment given above. The intent is to obtain that (a) a manifold acted upon by a proper Lie group action of a Lie group $G$ has an ``orbit-type stratification," (b) that the orbit-type stratification descends to a stratification of the orbit space $M/G$, and (c) the resulting stratum containing a point $[x]\in M/G$ is a manifold of dimension $\text{dim}(T_xS)^{G_x}$.

In particular, a point $x$ of a stratified tangent space lies in precisely one of the manifold strata, and so has a well-defined tangent space that can be associated to it, the \textit{stratified tangent space} to the point $x$, namely the tangent space of the point viewed as an element of the associated manifold stratum.

For our work, because the fixed point set of a vector space under a linear Lie group action (such as the tangent action of $G_x$) is itself a vector space, the stratified tangent space of a point $[x]\in M/G$ is isomorphic to the vector space $(T_xS)^{G_x}$. We now show that the stratified tangent space to a point of $M/G$ is isomorphic to the internal tangent space at that point.

\begin{thm}\label{ITS to STS}
    Let $M$ be a manifold acted upon by a proper Lie group action of a Lie group $G$, with $M/G$ the resulting orbit space. Then for $[x]\in M/G$ we have

    $$T_{[x]}(M/G) \cong (T_xS)^{G_x}$$

    That is, the internal tangent space to $[x]\in M/G$ viewed as a diffeological space with the quotient diffeology is isomorphic to the stratified tangent space to $[x]\in M/G$ viewed as a stratified space with induced orbit-type stratification.
\end{thm}

To prove this theorem, in addition to employing Definition \ref{def 3.1}, we also need to introduce the Haar measure of a compact Lie group and the related averaging principle, which we now discuss.

When a Lie group $G$ acts via a proper Lie group action on a manifold $M$, the stabilizer of the group $G_x$ at any point $x\in M$ is known to be compact. Therefore, $G_x$ possesses a unique, left-invariant Haar measure which maps $f\mapsto \int_{G_x}f(g)dg : C(G_x) \ra \cx$, such that $\int_{G_x}dg=1$ (this is discussed in \cite{dk_lie} section $4.2$; see also \cite{folland} section $11.1$). Here, $C(G_x)$ denotes the space of continuous $\mathbb{C}$-valued functions on $G_x$; such functions are automatically compactly supported.

The tangent action of $G_x$ is a compact, linear group action on the invariant vector subspace $T_xS$ of $T_xM$. Therefore, we can as in \cite{dk_lie} form the \textit{average map} $\Pi : T_xS \ra T_xS$ defined for $v\in T_xS$ by

$$\Pi(v)=\int_{G_x}g\cdot v \ dg = \int_{G_x}g_*(v) \ dg$$

\noindent where $g\in G_x$ acts on $v$ via the tangent action. In particular, $g_*$ is shorthand for $g_{*,x}$. As a linear map between vector spaces, it will be useful to think of this transformation as a matrix acting on Euclidean space.

Recall from the discussion between Theorem \ref{tube} and Proposition \ref{slice to isotropy} that as an invariant subspace of $T_xM$ we may regard $T_xS$ (and any invariant subspaces) as the Euclidean space $\rl^k$ for the appropriate $k$ with the standard Euclidean inner product, on which $G_x$ acts orthogonally. \textit{We make this association going forward}. We have the following proposition.

\begin{prop}[\cite{dk_lie} Proposition $4.2.1$; Averaging Principle]\label{avg}
    The map $\Pi$ is the orthogonal (with respect to the standard Euclidean inner product) linear projection from $T_xS$ to $(T_xS)^{G_x}$, the space of fixed points of the tangent action of $G_x$ on $T_xS$.
\end{prop}

\begin{rem}
    As described in \cite{dk_lie} on pages $215-216$, the left-invariance of the measure and properties of the integral imply $\Pi(g'\cdot v)=\Pi(v)$ for any $g'\in G_x$.
\end{rem}

We are now ready to begin the proof of Theorem \ref{ITS to STS}. Rather than overburden a single, beleaguered LaTeX proof environment we will proceed via several lemmas.

The idea is as follows. From Corollary \ref{main result cor} we know that $T_{[x_0]}(M/G) \cong T_{[0]}(T_xS/G_x)$. We now employ the first isomorphism theorem to construct an isomorphism from $(T_xS)^{G_x}$ to $T_{[0]}(T_xS/G_x)$. We begin by employing Definition \ref{def 3.1} which describes $T_{[0]}(T_xS/G_x)$ as $F/R$ where $F=\oplus_{\ol{p}}T_0(U_{\ol{p}})$ is the sum of tangent spaces at $0$ in the domains of all pointed plots $\ol{p}:U_{\ol{p}} \ra T_xS/G_x$ sending $0$ to $[0]$ in the quotient diffeology on $T_xS/G_x$ (which arises from the standard diffeology on the manifold $T_xS$) and $R$ is generated by the basic relations $(\ol{p},v) - (\ol{q}, h_*(v))$ where $h$ represents a germ of smooth maps with $\ol{p} = \ol{q} \circ h$ as germs at $0$ (we use $h$ here rather than $g$ for an arbitrary germ of smooth maps to avoid confusion with Lie group elements). We then define the map $\varphi : T_xS \ra F/R$ by:

$$\varphi(v) = (\pi, v)\in F/R$$

\noindent for $v\in T_xS$. This map warrants explanation. The map $\pi$ (here and below representing $\pi : T_xS \ra T_xS/G_x$), the projection operator associated to the tangent action of $G_x$, is itself a plot in the quotient diffeology on $T_xS/G_x$ (as $\pi = \pi \circ id$, where the identity map $id:T_xS \ra T_xS$ lies in the standard diffeology of $T_xS$). Therefore, $v\in T_xS$ lies in the domain of a plot in the quotient diffeology and so can be identified as an element of $F$ in the summand $T_0(U_{\pi})=T_0(T_xS)=T_xS$. Therefore, with $\iota$ as the inclusion map $T_xS \hookrightarrow T_0(T_xS)=T_0(U_{\pi}) \subset F$ and $\ol{\pi}$ the projection $F \ra F/R$, we have $\varphi = \ol{\pi} \circ \iota$ and the diagram

$$
\xymatrix{
T_xS ~ \ar@{^{(}->}[r]^{\iota} \ar[dr]_{\varphi} & F \ar[d]^{\ol{\pi}} \\
& F/R
}
$$

Because $\ol{\pi}$ is a linear projection ($F$ is a vector space and $F/R$ its vector space quotient), $\varphi$ is clearly a homomorphism. We can also easily show $\varphi$ is surjective. Indeed, for $(\ol{p},v)$ any element of $F/R$ the fact that $\ol{p}$ is in the quotient diffeology means that locally about $0$ in the domain of $\ol{p}$ we have $\ol{p} = \pi \circ p$ for $p$ a smooth map in the standard diffeology on $T_xS$. We have

$$
\xymatrix{
U_{\ol{p}} \ar[rr]^{p} \ar[dr]_{\ol{p}} & & T_xS \ar[dl]^{\pi} \\
& T_xS/G_x &
}
$$

\noindent and hence in $F/R$ we have $(\ol{p}, v) \sim (\pi, p_*(v))$ (i.e., $(\ol{p}, v) - (\pi, p_*(v))\in R$). This implies that any internal tangent vector $(\ol{p},v)\in F/R$ is equivalent to an element which takes a value of zero outside the summand $T_0(T_xS)$ corresponding to plot $\pi$. As $\varphi$ maps $T_xS =T_0(T_xS)$ by inclusion onto this summand term in $F$ and then projects onto $F/R$, it is surjective.

We now turn to the kernel of our surjective homomorphism $\varphi$. First, we need to decompose the domain $T_xS$ into two orthogonal subspaces. Recall that the fixed point set $(T_xS)^{G_x}$ is a closed subspace of $T_xS$, due to the fact that the tangent action is continuous and linear. Therefore, we can write $T_xS = (T_xS)^{G_x} \oplus ((T_xS)^{G_x})^{\perp}$ and we can decompose any element $v\in T_xS$ uniquely as $v=v_f + v_o$ for some $v_f \in (T_xS)^{G_x}$ and $v_o \in ((T_xS)^{G_x})^{\perp}$.

Now, our primary goal is to show that the kernel of $\varphi$ is the perpendicular space, $\text{ker}(\varphi) = ((T_xS)^{G_x})^{\perp}$. Then, by the first isomorphism theorem, our surjective homomorphism descends to an isomorphism from $(T_xS)^{G_x}$ to $F/R$. Combined with Corollary \ref{main result cor} and the fact that $F/R$ is $T_{[0]}(T_xS/G_x)$, the proof will then be complete. Our first substantial step toward this goal is the following lemma.

\begin{lem}\label{ITS to STS lem 1}
    The kernel of $\varphi$ contains $((T_xS)^{G_x})^{\perp}$. That is, $\text{ker}(\varphi) \supseteq ((T_xS)^{G_x})^{\perp}$.
\end{lem}

\begin{proof}
    For this proof, let $P$ denote $((T_xS)^{G_x})^{\perp}$. Suppose we have an element $v_o\in P \subseteq T_xS$. We aim to show that $v_o\in \text{ker}(\varphi)=\text{ker}(\ol{\pi}\circ \iota)$. That is, we need to display $\iota(v_o)=(\pi,v_o)$ as a linear combination of basic relations $(\ol{p},v) - (\ol{q},h_*(v))$ in $R$.

    Indeed, because $v_o\in P$, the average map $\Pi$ described above and detailed in Proposition \ref{avg} maps $v_o$ to $0$ in $T_xS$, i.e., $\Pi(v_o)=0$. We can apply $\Pi$ to the summand $T_0(U_{\pi})=T_0(T_xS)=T_xS$ and interpret the above as $\Pi((\pi,v_o))=0$ in $F$ (we will also employ this below); that is (recalling that here and below $g_{*,x}:=g_*$),

    \begin{align}\label{ITS to STS lem 1 eq}
    \int_{G_x} g_*((\pi,v_o)) \ dg = 0
    \end{align}

    We now show why the above equation allows us to write

    $$(\pi,v_o) = \sum c_i[(\ol{p_i},v_i)-(\ol{q_i},h^i_*(v_i))]$$

    \noindent for $c_i\in \rl$, pointed plots $\ol{p_i}, \ol{q_i}$, germs of smooth maps $h^i$, and such that $\ol{p_i} = \ol{q_i}\circ h^i$ as germs (where each $v_i$ represents an arbitrary element of the tangent space $T_0(U_{\ol{p_i}})$). Indeed, naively, suppose we could obtain a linearly dependent set of vectors of the form $\{(\pi,g^i_*v_o)\}_{i=1}^N$ (that is, $\sum_{i=1}^Nc_i(\pi, g^i_*v_o)=0$ for come constants $c_i$ not all zero), where each $g^i\in G_x$. Then by taking $C:=\sum c_i$ we would have

    $$C(\pi, v_o) = C(\pi,v_o) - \sum_{i=1}^Nc_i(\pi, g^i_*v_o)$$

    \noindent and the argument would be complete upon dividing by $C$. However, there is a flaw in that we have disregarded the possibility that $C=0$, which must be remedied.

    We therefore proceed more carefully, breaking the argument into two cases. First, consider the situation when $|G_x|$ is finite. In this case, the statement that $\Pi((\pi,v_o))=0$ reduces to

    $$(\pi,g^1_*v_o) + \cdots + (\pi,g^N_*v_o)=0$$

    \noindent for the finitely many $g^i\in G_x$ and where $N=|G_x|$. Thus in this case we can write

    $$|G_x|(\pi, v_o) = |G_x|(\pi, v_o) - \sum_{i=1}^{|G_x|}(\pi, g^i_*v_o)$$

    \noindent and because $\pi = \pi \circ g^i$ for all $g^i\in G_x$ (for $\pi: T_xS \ra T_xS/G_x$) along with the fact that $0<|G_x|<\infty$ this shows that $(\pi,v_o)$ lies in $R$ and hence $v_o$ lies in $\text{ker}(\varphi)$.

    In the case where $|G_x|$ is infinite, begin by letting $\{(\pi,g^i_*v_o)\}_{i=1}^n$ denote a maximal linearly independent set in $P$ taken from vectors of the form $g^i_*v_o$ with $g^i\in G_x$. Because all such vectors lie in the finite-dimensional vector space $T_xS$, this is allowed. Then, for any $g\in G_x$, there are constants $c_i$ and $c$ (not all zero) such that

    $$\sum_{i=1}^nc_i(\pi, g^i_*v_o) + c(\pi, g_*v_o)=0$$

    \noindent which can be obtained by solving a linear system. We aim to solve the system with the added constraint that $(\sum_{i=1}^nc_i) + c = 1$ (equivalently, $(\sum_{i=1}^nc_i) + c \neq 0$). Consider the resulting augmented system (wherein we use vector notation for clarity):

    $$
    \begin{amatrix}{4}
    \ora{g^1_*v_o} & \cdots & \ora{g^n_*v_o} & \ora{g_*v_o} & \vec{0} \\
     &  &  &  &  \\
    1 & \cdots & 1 & 1 & 1 \\
    \end{amatrix}
    $$

    In this system, the $\ora{g^1_*v_o}, \dots, \ora{g^n_*v_o}$ constitute an $m\times n$ submatrix with $m\ge n$ and $m=\text{dim}(T_xS)$. The vectors $\ora{g_*v_o}$ and $\vec{0}$ are both $m\times 1$. The bottom row of the matrix consists solely of $1$s. Overall, then, the system is $(m+1)\times(n+1)$, not counting the augmented column. We proceed to reduce the system as follows.

    \begin{enumerate}
        \item First, perform elementary row operations to reduce the $m\times n$ submatrix containing $\ora{g^1_*v_o}, \dots, \ora{g^n_*v_o}$ to
        $$
        \begin{pmatrix}
            Id_{n\times n} \\
            0_{(m-n)\times n}
        \end{pmatrix}
        $$
        \noindent where $Id$ denotes the identity matrix (of dimension $n$) and below it lie $m-n$ rows of $0$s. This can be done, as the vectors composing this submatrix are linearly independent. Note that this can be done independent of the column containing $\ora{g_*v_o}$ (although this column is potentially altered by the process) and without affecting the bottom row of $1$s in the overall matrix.
        \item Second, eliminate all the $1$s in the bottom row via subtraction, except the final $1$ in the column containing $\ora{g_*v_o}$.
        \item Now, the final $(m+1)$th entry of the column which originally contained $\ora{g_*v_o}$ takes the form
        $$1-\alpha_1(g_*v_o)_1 - \cdots - \alpha_m(g_*v_o)_m$$
        \noindent where $(g_*v_o)_i$ denotes the $i$th entry of $\ora{g_*v_o}$ as an element of $T_xS$ and the $\alpha_i$'s are real constants independent of $\ora{g_*v_o}$ (they depend only on the $\ora{g^i_*v_o}$ vectors and how they were row reduced). In particular, we can vary $g$ in $\ora{g_*v_o}$ without varying the $\alpha_i$'s.
        \item If the entry referenced in the item above satisfies
        $$1-\alpha_1(g_*v_o)_1 - \cdots - \alpha_m(g_*v_o)_m \neq 0$$
        \noindent then the overall augmented matrix can be solved, meaning that we have attained our desired linear dependence relation with the additional constraint on the constants; namely, that $(\sum_{i=1}^nc_i) + c = 1 \neq 0$. That there exists a $g$ such that $\ora{g_*v_o}$ allows for this is justified below.
    \end{enumerate}

    Indeed, the condition that $1-\alpha_1(g_*v_o)_1 - \cdots - \alpha_m(g_*v_o)_m \neq 0$ is equivalent to $\langle\ora{\alpha}, \ora{g_*v_o}\rangle\neq 1$, where $\langle\cdot,\cdot\rangle$ denotes the Euclidean inner product on $T_xS$. We show there exists such a $g\in G_x$ by contradiction. Suppose to the contrary that for all $g\in G_x$ we have
    $$\langle\ora{\alpha}, \ora{g_*v_o}\rangle = 1$$
    \noindent where here we are taking full advantage of the fact that $\ora{\alpha}$ is independent of $g$ in $\ora{g_*v_o}$. In this case, the fact that $\Pi(v_o)=0$ and the fact that integral and inner product can be interchanged together imply

    $$0=\langle\ora{\alpha},\Pi(v_o)\rangle=\int_{G_x}\langle\ora{\alpha},\ora{g_*(v_o)}\rangle \ dg = \int_{G_x}1 \ dg = 1$$

    This is a contradiction. Therefore, there exists a $g\in G_x$ such that

    $$
    \begin{amatrix}{4}
    \ora{g^1_*v_o} & \cdots & \ora{g^n_*v_o} & \ora{g_*v_o} & \vec{0} \\
     &  &  &  &  \\
    1 & \cdots & 1 & 1 & 1 \\
    \end{amatrix}
    $$

    \noindent can be solved. Let $c_1, \dots, c_n, c$ denote the entries of the solution vector. We have shown, with $C=(\sum_{i=1}^nc_i)+c$, that

    $$C(\pi, v_o) = C(\pi,v_o) - \sum_{i=1}^nc_i(\pi, g^i_*v_o)-(\pi,g_*v_o)$$

    \noindent and, because we have $\pi = \pi \circ g^i=\pi\circ g$ (for $\pi:T_xS \ra T_xS/G_x$ and $g^i,g\in G_x$) and the fact that $C\neq 0$, we have shown that $(\pi,v_o)\in R$; hence $v_o\in \text{ker}(\varphi)$.

    Thus, whether $|G_x|$ is finite or infinite, we have shown that for any $v_o\in P$ it is the case that $\iota(v_o)=(\pi,v_o)$ lies in $R$, hence $v_o\in \text{ker}(\varphi)$. This shows $\text{ker}(\varphi) \supseteq ((T_xS)^{G_x})^{\perp}$, as desired.
\end{proof}

\begin{ex}
    We demonstrate solving the system in the above lemma in the case of $\rl^3/SO(3)$, when seeking to determine $T_{[0]}(\rl^3/SO(3))$. In this case the stabilizer of the point $0$ satisfies $SO(3)_{0}=SO(3)$ and we can take all $\rl^3$ as a slice through $0$. Therefore, $T_xS=\rl^3$ and $(T_xS)^{SO(3)}=\{0\}$, meaning that $((T_xS)^{SO(3)})^{\perp}=\rl^3$ as well. We consider $v_o=(1,1,1)^T$ and the maximal (in $T_0(\rl^3)=\rl^3$) linearly independent set: 

    $$
    \begin{pmatrix}
        -1 \\
        1 \\
        -1
    \end{pmatrix}, 
    \begin{pmatrix}
        -1 \\
        -1 \\
        1
    \end{pmatrix},
    \begin{pmatrix}
        1 \\
        -1 \\
        -1
    \end{pmatrix}
    $$

    \noindent arising from $v_o$ via the elements $g^1,g^2,g^3$ of $SO(3)=SO(3)_0$ given as matrices by:

    $$ 
    \begin{pmatrix}
        -1 & 0 & 0 \\
        0 & 1 & 0 \\
        0 & 0 & -1
    \end{pmatrix}, \ \
    \begin{pmatrix}
        -1 & 0 & 0 \\
        0 & -1 & 0 \\
        0 & 0 & 1
    \end{pmatrix}, \ \
    \begin{pmatrix}
        1 & 0 & 0 \\
        0 & -1 & 0 \\
        0 & 0 & -1
    \end{pmatrix}
    $$

    We now consider an arbitrary $g\in SO(3)$ and resulting vector $g_*v_o:=(x,y,z)^T$, and proceed to solve the system which arises as in Lemma \ref{ITS to STS lem 1}. Here, it takes the form

    $$
    \begin{amatrix}{4}
    -1 & -1 & 1 & x & 0 \\
    1 & -1 & -1 & y & 0\\
    -1 & 1 & -1 & z & 0 \\
    1 & 1 & 1 & 1 & 1 \\
    \end{amatrix}
    $$

    After row-reducing to the identity submatrix (here there are no rows of zeros), we have

    $$
    \begin{amatrix}{4}
    1 & 0 & 0 & -\frac{x}{2}-\frac{z}{2} & 0 \\
    0 & 1 & 0 & -\frac{x}{2}-\frac{y}{2} & 0 \\
    0 & 0 & 1 & -\frac{y}{2}-\frac{z}{2} & 0 \\
    1 & 1 & 1 & 1 & 1 \\
    \end{amatrix}
    $$

    \noindent whereupon after eliminating the $1$s at the bottom of the first three columns we obtain

    $$
    \begin{amatrix}{4}
    1 & 0 & 0 & -\frac{x}{2}-\frac{z}{2} & 0 \\
    0 & 1 & 0 & -\frac{x}{2}-\frac{y}{2} & 0 \\
    0 & 0 & 1 & -\frac{y}{2}-\frac{z}{2} & 0 \\
    0 & 0 & 0 & 1 +x+y+z & 1 \\
    \end{amatrix}
    $$

    Here, the system is solvable, provided (as laid out in Lemma \ref{ITS to STS lem 1}):
    $$1 - (-1)x - (-1)y - (-1)z \neq 0$$
    \noindent where $\ora{\alpha} = (-1,-1,-1)^T$ is independent of $g_*v_o=(x,y,z)^T$. Indeed, in this instance we can take $g=id:\rl^3 \ra \rl^3$, the identity map, so that $g_*v_o=v_o$ and then $1 + 1 + 1+ 1 \neq 0$. Thus, the system is indeed solvable (as proven indirectly in the lemma) in this case.
\end{ex}

Returning to the proof of Theorem \ref{ITS to STS}, so far, we know the following (even before knowing the exact nature of the kernel and image).

$$\text{dim}(T_xS)^{G_x} + \text{dim}((T_xS)^{G_x})^{\perp} = \text{dim}T_xS = \text{dim(ker}(\varphi)) + \text{dim(im}(\varphi))$$

Based on the above lemma, we can then infer

$$\text{dim}(T_xS)^{G_x} - \text{dim(im}(\varphi))  = \text{dim(ker}(\varphi)) - \text{dim}((T_xS)^{G_x})^{\perp} \ge 0$$

Hence, if we can show $0\ge \text{dim}(T_xS)^{G_x} - \text{dim(im}(\varphi))$ then it follows that $\text{ker}(\varphi) = ((T_xS)^{G_x})^{\perp}$, as desired. In light of that fact that any $v\in T_xS$ has a unique decomposition of the form $v=v_f + v_o$ for $v_f\in (T_xS)^{G_x}$ and $v_o\in ((T_xS)^{G_x})^{\perp}$ and that $\varphi$ is a homomorphism, the following lemma achieves precisely this result.

\begin{lem}\label{ITS to STS lem 2}
    The map $\varphi:T_xS \ra F/R$ restricted to $(T_xS)^{G_x}$ is injective.
\end{lem}

Before proving this result, we need another technical lemma. We reintroduce some concepts before stating it. First, recall the average map $\Pi$ introduced before and within Proposition \ref{avg}. It will be necessary below to know the nature of its differential (at $0$).

Secondly, recall that we are of conceiving the internal tangent space $T_{[0]}(T_xS/G_x)$ as $F/R$ where $F=\oplus_{\ol{p}}T_0(U_{\ol{p}})$ is the sum of tangent spaces at $0$ in the domains of all pointed plots $\ol{p}:U_{\ol{p}} \ra T_xS/G_x$ sending $0$ to $[0]$ in the quotient diffeology on $T_xS/G_x$ (arising from the standard diffeology on $T_xS$) and $R$ is generated by the basic relations $(\ol{p},v) - (\ol{q}, h_*(v))$ where $h$ represents a germ of smooth maps with $\ol{p} = \ol{q} \circ h$ as germs at $0$.

In particular, for any pointed plot $\ol{p} : U_{\ol{p}} \ra T_xS/G_x$ we have (because the diffeology here is the quotient diffeology) that $\ol{p}=\pi \circ p$ as germs for $\pi$ the projection from $T_xS$ to $T_xS/G_x$ and $p:U_{\ol{p}} \ra T_xS$ a smooth map in the usual sense. In diagram form:

$$
\xymatrix{
U_{\overline{p}} \ar[rr]^{p} \ar[dr]_{\overline{p}} & & T_xS \ar[dl]^{\pi} \\
& T_xS/G_x &
}
$$

However, it is possible that $p$ is not unique. That is, we could also have $\ol{p}=\pi \circ p'$ as germs for $p'$ a smooth map distinct from $p$. Nonetheless, a statement related to the differentials at zero can still be made. 

\begin{lem}\label{ITS to STS lem tech}
    Let $\Pi$ be the average map introduced in and before Proposition \ref{avg}, with domain $T_xS$ (equivalently, the summand of $F$ corresponding to plot $\pi:T_xS \ra T_xS / G_x$).
    
    Also consider $\ol{p}:U_{\ol{p}} \ra T_xS/G_x$ an arbitrary pointed plot in the quotient diffeology of $T_xS/G_x$ and suppose that $\ol{p} = \pi \circ p = \pi \circ p'$ as germs at $0$, for $p, p'$ smooth maps $U_{\ol{p}} \ra T_xS$.
    
    Lastly suppose we have two pointed plots $\ol{p}, \ol{q}$ such that $\ol{p}=\pi \circ p$, $\ol{q}=\pi \circ q$ as germs at $0$ for $p,q$ appropriate smooth maps, and $\ol{p} = \ol{q} \circ h$ as germs at $0$ for $h$ the germ of a smooth map; this means that $(\ol{p},v)-(\ol{q},h_*v) \in R$ for $v\in T_0(U_{\ol{p}})$.
    
    Here and below, let $_*$ be shorthand for $_{*,0}$ the differential of a map at $0$. We have the following three facts:
    \begin{enumerate}
        \item The map $\Pi$ is its own differential at $0$. That is, $\Pi = \Pi_*$.
        \item From $\ol{p} = \pi \circ p = \pi \circ p'$, we have $\Pi_* \circ p_* = \Pi_* \circ p'_*$.
        \item From $\ol{p} = \ol{q} \circ h$, we have $\Pi_*\circ p_*=\Pi_* \circ q_* \circ h_*$
    \end{enumerate}
    In effect, related plots have equal differentials after composing with $\Pi$. The second fact, basically a well-definedness statement, ensures that the third is meaningful.
\end{lem}

\begin{proof}
    The first item is immediate. Indeed, because $\Pi$ is a linear map on a finite dimensional vector space $T_xS$ it can be identified with its differential $\Pi_*$.

    For the second item, suppose we have $\ol{p}=\pi \circ p = \pi \circ p'$ as germs about $0$ as described. Then for any point $w$ in a neighborhood of $0$ we have

    $$p(w) = g_w \cdot p'(w)$$

    \noindent for $g_w$ an element of $G_x$ (which need not vary smoothly with $w$). Applying $\Pi$ to both sides, we have from the remark following Proposition \ref{avg} that
    
    $$\Pi(p(w)) = \Pi(g_w\cdot p'(w))=\Pi(p'(w))$$

    As this holds for all $w$ in a neighborhood of zero, by the chain rule $\Pi_* \circ p_* = \Pi_* \circ p'_* $.

    For the third item, suppose we have $\ol{p}=\pi \circ p$, $\ol{q}=\pi \circ q$, and $\ol{p} = \ol{q} \circ h$ as germs about $0$ as described. Then the last equality implies that in a neighborhood of $0$ we have

    $$\pi \circ p = \pi \circ q \circ h$$

    \noindent meaning that for any $w$ in a neighborhood of $0$ it is the case that

    $$p(w) = g_w \cdot q(h(w))$$

    Again, $g_w\in G_x$ and can vary non-smoothly with $w$. We apply $\Pi$ as done above and obtain $\Pi(p(w)) = \Pi(q(h(w)))$ for all $w$ in a neighborhood of $0$. Thus, we have by the chain rule that $\Pi_* \circ p_* = \Pi_* \circ q_* \circ h_*$.
\end{proof}

We now return to the proof of Lemma \ref{ITS to STS lem 2} pursuant to the proof of Theorem \ref{ITS to STS}.

\begin{proof}[Proof of Lemma \ref{ITS to STS lem 2}]
    We continue to think of the internal tangent space $T_{[0]}(T_xS/G_x)$ as the quotient $F/R$. Suppose $v_f\in (T_xS)^{G_x}$ and $v_f\in \text{ker}(\varphi)$, meaning that $\ol{\pi}(\iota(v_f))=0$ for $\iota : T_xS \ra F$ the inclusion map and $\ol{\pi} : F \ra F/R$ the projection map. This implies that $\iota(v_f)\in R$, meaning it can be written as a finite sum

    \begin{align}\label{ITS to STS lem 2 eq}
        \iota(v_f)=(\pi,v_f)=\sum_i c_i[(\ol{p_i},v_i)-(\ol{q_i},h^i_*v_i)]
    \end{align}

    \noindent where $\pi : T_xS \ra T_xS/G_x$, $c_i\in \rl$, $v_i\in T_0(U_{\ol{p_i}})$, the indexes $\ol{p_i},\ol{q_i}$ are pointed plots, and each $h^i$ is an arbitrary germ of smooth maps such that $\ol{p_i}=\ol{q_i}\circ h^i$ as germs about $0$. Our goal is to show that $v_f=0$, for which it is sufficient to show that $\iota(v_f)=(\pi, v_f)=(\pi,0)$.

    To do this, we begin with some observations drawn from equality \ref{ITS to STS lem 2 eq}. First, note that because $v_f\in (T_xS)^{G_x}$ we have $\Pi(v_f)=v_f$, by Proposition \ref{avg}. Now, because equality \ref{ITS to STS lem 2 eq} holds in the space $F$ before modding out by $R$, we can assert two useful facts:
    \begin{enumerate}
        \item (First Observation) First, although the right-hand side of the equation contains elements from possibly many plot indices of the direct sum $F$, because the left-hand side contains only the index $\pi$, any terms on the right-hand side indexed by a plot other than $\pi$ must cancel out completely with other elements on the right-hand side (indexed by the same plot) or be zero.\label{observation 1}
        \item (Second Observation) Second, because $\Pi(v_f)=v_f$, any term indexed by $\pi$ on the right hand side, that is, a term of the form $(\pi, v_i)$ or $(\pi,h^i_*v_i)$, can be replaced by the term $(\pi,\Pi(v_i))$ or $(\pi,\Pi(h^i_*v_i))$, respectively. This is because the left-hand side consists solely of a term in the image of $\Pi$, so any terms or components on the right hand side not in this image must cancel completely with other elements on the right hand side or be zero.\label{observation 2}
    \end{enumerate}

    As shorthand, we will refer to the above two notions as the First and Second Observations (derived from equation \ref{ITS to STS lem 2 eq}). Now, recall that for any pointed plot $\ol{p_i}$ or $\ol{q_i}$ in the quotient diffeology we have $\ol{p_i}=\pi \circ p_i$ or $\ol{q_i}=\pi \circ q_i$ for $p_i$ or $q_i$ a smooth map into $T_xS$ in the usual sense.
    
    With this in mind, we expand the expression from equation \ref{ITS to STS lem 2 eq} to the below:
    
    \begin{align*}
        & \quad \sum_i c_i[(\ol{p_i},v_i)-(\ol{q_i},h^i_*v_i)] \\
        &= \sum_i c_i[(\ol{p_i},v_i)-(\pi,p_{i*}v_i)+(\pi,p_{i*}v_i)-(\pi,q_{i*}h^i_*v_i)+(\pi,q_{i*}h^i_*v_i)-(\ol{q_i},h^i_*v_i)]
    \end{align*}
    
    Now, rewrite the expanded expression as

    \begin{align}\label{regrouped 1}
    \sum_i c_i[(\pi,p_{i*}v_i)-(\pi,q_{i*}h^i_*v_i)] + \sum_i c_i[(\ol{p_i},v_i)-(\pi,p_{i*}v_i)+(\pi,q_{i*}h^i_*v_i)-(\ol{q_i},h^i_*v_i)]
    \end{align}

    From here, group related pairs in the second sum above not by index, but rather by plot. That is, whenever $\ol{p_i}$ agrees with some other $\ol{p_j}$ or $\ol{q_k}$ (and likewise for plots $\ol{q_i}$), group all such pairs $(\ol{p_i},v_i)-(\pi,p_{i*}v_i)$ or $(\pi,q_{i*}h^i_*v_i)-(\ol{q_i},h^i_*v_i)$ together. With this organization, where we label groups of the same plot using only $\ol{p}$'s (no longer $\ol{q}$'s), expression \ref{regrouped 1} becomes:

    \begin{align*}
        (\pi, v_f) &= \sum_i c_i[(\pi,p_{i*}v_i)-(\pi,q_{i*}h^i_*v_i)] \\
        &+ \sum_{j=1}^{m_1}c^1_j[(\ol{p_1},v^1_j)-(\pi,p_{1*}v^1_j)] + \sum_{k=1}^{n_1} c^1_k[(\pi,p_{1*}(h^{1,k})_*v^1_k)-(\ol{p_1},(h^{1,k})_*v^1_k)] \\
        & \vdots \\
        &+ \sum_{j=1}^{m_N} c^N_j[(\ol{p_N},v^N_j)-(\pi,p_{N*}v^N_j)] + \sum_{k=1}^{n_N} c^N_k[(\pi,p_{N*}(h^{N,k})_*v^N_k)-(\ol{p_N},(h^{N,k})_*v^N_k)]
    \end{align*}

    \noindent where it is possible that $m_i$ or $n_i$ may be zero. We now proceed to show the above collection of sums equals zero. First, note that because $\Pi(v_f) = v_f$, by our Second Observation \ref{observation 2} from the beginning of the proof all terms in our grouped expression indexed by the plot $\pi$ must lie in the image of $\Pi$. Therefore, we may replace the above with:

    \begin{align*}
        & \quad (\pi, v_f) \\
        &= (\pi, \Pi(v_f)) \\
        &= \sum_i c_i[(\pi,\Pi(p_{i*}v_i))-(\pi,\Pi(q_{i*}h^i_*v_i))] \\
        &+ \sum_{j=1}^{m_1} c^1_j[(\ol{p_1},v^1_j)-(\pi,\Pi(p_{1*}v^1_j))] + \sum_{k=1}^{n_1} c^1_k[(\pi,\Pi(p_{1*}(h^{1,k})_*v^1_k))-(\ol{p_1},(h^{1,k})_*v^1_k)] \\
        & \vdots \\
        &+ \sum_{j=1}^{m_N} c^N_j[(\ol{p_N},v^N_j)-(\pi,\Pi(p_{N*}v^N_j))] + \sum_{k=1}^{n_N} c^N_k[(\pi,\Pi(p_{N*}(h^{N,k})_*v^N_k))-(\ol{p_N},(h^{N,k})_*v^N_k)]
    \end{align*}

    Observe that the first and third parts of Lemma \ref{ITS to STS lem tech} imply

    $$\sum_i c_i[(\pi,\Pi(p_{i*}v_i))-(\pi,\Pi(q_{i*}h^i_*v_i))]=\sum c_i[(\pi,\Pi_*p_{i*}v_i)-(\pi,\Pi_*p_{i*}v_i)]=0$$

    \noindent and so this entire sum (the first part of our extended expression) can be eliminated from the expression for $(\pi,v_f)$. This leaves us with

    \begin{align*}
        & \quad (\pi, v_f) \\
        &= (\pi, \Pi(v_f)) \\
        &= \sum_{j=1}^{m_1} c^1_j[(\ol{p_1},v^1_j)-(\pi,\Pi(p_{1*}v^1_j))] + \sum_{k=1}^{n_1} c^1_k[(\pi,\Pi(p_{1*}(h^{1,k})_*v^1_k))-(\ol{p_1},(h^{1,k})_*v^1_k)] \\
        & \vdots \\
        &+ \sum_{j=1}^{m_N} c^N_j[(\ol{p_N},v^N_j)-(\pi,\Pi(p_{N*}v^N_j))] + \sum_{k=1}^{n_N} c^N_k[(\pi,\Pi(p_{N*}(h^{N,k})_*v^N_k))-(\ol{p_N},(h^{N,k})_*v^N_k)]
    \end{align*}
    
    Now, observe that if $\ol{p_\ell} = \pi$ for one of our $N$ remaining groupings, the grouping takes the form:

    \begin{align}\label{regrouped 2}
    \sum_{j=1}^{m_\ell} c^\ell_j[(\pi,\Pi(v^\ell_j))-(\pi,\Pi(p_{\ell*}v^\ell_j))] + \sum_{k=1}^{n_\ell} c^\ell_k[(\pi,\Pi(p_{\ell*}(h^{\ell,k})_*v^\ell_k))-(\pi,\Pi((h^{\ell,k})_*v^\ell_k))]
    \end{align}

    \noindent where $\pi = \ol{p_\ell} = \pi \circ p_\ell$. However, this can be further simplified in light of Lemma \ref{ITS to STS lem tech}. Indeed, because we also have $\ol{p_\ell} = \pi = \pi \circ id$, where $id:T_xS \ra T_xS$ denotes the identity map, it follows from the first and second parts of the lemma that $\Pi p_{\ell*} = \Pi_*p_{\ell*} = \Pi_*id_*=\Pi_*$. Therefore, since Lemma \ref{ITS to STS lem tech} also gives that $\Pi=\Pi_*$, expression \ref{regrouped 2} becomes

    $$\sum_{j=1}^{m_\ell} c^\ell_j[(\pi,\Pi_*v^\ell_j)-(\pi,\Pi_*v^\ell_j)] + \sum_{k=1}^{n_\ell} c^\ell_k[(\pi,\Pi_*(h^{\ell,k})_*v^\ell_k)-(\pi,\Pi_*(h^{\ell,k})_*v^\ell_k)] = 0$$

    Therefore, such a grouping can be eliminated. We have so far shown that for any $v_f\in (T_xS^{G_x})$, $(\pi,v_f)$ can be expressed as

    \begin{align*}
        & \quad (\pi, v_f) \\
        &= (\pi, \Pi(v_f)) \\
        &= \sum_{j=1}^{m_1} c^1_j[(\ol{p_1},v^1_j)-(\pi,\Pi(p_{1*}v^1_j))] + \sum_{k=1}^{n_1} c^1_k[(\pi,\Pi(p_{1*}(h^{1,k})_*v^1_k))-(\ol{p_1},(h^{1,k})_*v^1_k)] \\
        & \vdots \\
        &+ \sum_{j=1}^{m_N} c^N_j[(\ol{p_N},v^N_j)-(\pi,\Pi(p_{N*}v^N_j))] + \sum_{k=1}^{n_N} c^N_k[(\pi,\Pi(p_{N*}(h^{N,k})_*v^N_k))-(\ol{p_N},(h^{N,k})_*v^N_k)]
    \end{align*}

    \noindent where $\ol{p_\ell}\neq \pi$. We now show that these remaining terms can also be eliminated. Indeed, for such a grouping associated to a plot $\ol{p_\ell}$, namely

    \begin{align*}
        & \quad \sum_{j=1}^{m_\ell} c^\ell_j[(\ol{p_\ell},v^\ell_j)-(\pi,\Pi(p_{\ell*}v^\ell_j))] + \sum_{k=1}^{n_\ell} c^\ell_k[(\pi,\Pi(p_{\ell*}(h^{\ell,k})_*v^\ell_k))-(\ol{p_\ell},(h^{\ell,k})_*v^\ell_k)] \\
        &= \left(\ol{p_\ell},\sum_{j=1}^{m_\ell} c^\ell_jv^\ell_j - \sum_{k=1}^{n_\ell}c^\ell_k(h^{\ell,k})_*v^\ell_k\right) + \left(\pi,-\sum_{j=1}^{m_\ell} c^\ell_j\Pi(p_{\ell*}v^\ell_j) + \sum_{k=1}^{n_\ell}c^\ell_k\Pi(p_{\ell*}(h^{\ell,k})_*v^\ell_k)\right)
    \end{align*}
    
    \noindent we must have that the first term in the second line above satisfies

    \begin{align}\label{regrouped 3}
    \sum_{j=1}^{m_\ell} c^\ell_j(\ol{p_\ell},v^\ell_j) - \sum_{k=1}^{n_\ell} c^\ell_k(\ol{p_\ell},(h^{\ell,k})_*v^\ell_k) = \left(\ol{p_\ell},\sum_{j=1}^{m_\ell} c^\ell_jv^\ell_j - \sum_{k=1}^{n_\ell}c^\ell_k(h^{\ell,k})_*v^\ell_k\right) = 0
    \end{align}

    \noindent because $\ol{p_\ell}\neq \pi$ and, as noted in our First Observation \ref{observation 1} at the start of this proof, such terms must cancel completely (and only with other terms indexed by the same plot $\ol{p_\ell}$). Now, as all $v^\ell_j$ and $(h^{\ell,k})_*v^\ell_k$ lie in $T_0(U_{\ol{p_\ell}})$ for all $j,k$, we can apply the differential $\Pi_*p_{\ell*}$ to this collection of vectors to observe that

    \begin{align}
    0 &= -1\cdot\Pi_*p_{\ell*}\left[\sum_{j=1}^{m_\ell} c^\ell_jv^\ell_j - \sum_{k=1}^{n_\ell}c^\ell_k(h^{\ell,k})_*v^\ell_k\right] \\
    &= -\sum_{j=1}^{m_\ell}c^\ell_j\Pi_*p_{\ell*}v^\ell_j + \sum_{k=1}^{n_\ell}c^\ell_k\Pi_*p_{\ell*}(h^{\ell,k})_*v^\ell_k \label{regrouped 4}
    \end{align}

    From the above equations \ref{regrouped 3} and \ref{regrouped 4}, along with the first part of Lemma \ref{ITS to STS lem tech}, it follows that

    \begin{align*}
        & \quad \sum_{j=1}^{m_\ell} c^\ell_j[(\ol{p_\ell},v^\ell_j)-(\pi,\Pi(p_{\ell*}v^\ell_j))] + \sum_{k=1}^{n_\ell} c^\ell_k[(\pi,\Pi(p_{\ell*}(h^{\ell,k})_*v^\ell_k))-(\ol{p_\ell},(h^{\ell,k})_*v^\ell_k)] \\
        &= \left(\ol{p_\ell},\sum_{j=1}^{m_\ell} c^\ell_jv^\ell_j - \sum_{k=1}^{n_\ell}c^\ell_k(h^{\ell,k})_*v^\ell_k\right) + \left(\pi,-\sum_{j=1}^{m_\ell} c^\ell_j\Pi(p_{\ell*}v^\ell_j) + \sum_{k=1}^{n_\ell}c^\ell_k\Pi(p_{\ell*}(h^{\ell,k})_*v^\ell_k)\right) \\
        &=0
    \end{align*}

    \noindent as desired. Hence each remaining term in the expression for $(\pi,v_f)$ can be eliminated. This means that overall

    \begin{align*}
        & \quad (\pi, v_f) \\
        &= (\pi, \Pi(v_f)) \\
        &= \sum_{j=1}^{m_1} c^1_j[(\ol{p_1},v^1_j)-(\pi,\Pi(p_{1*}v^1_j))] + \sum_{k=1}^{n_1} c^1_k[(\pi,\Pi(p_{1*}(h^{1,k})_*v^1_k))-(\ol{p_1},(h^{1,k})_*v^1_k)] \\
        & \vdots \\
        &+ \sum_{j=1}^{m_N} c^N_j[(\ol{p_N},v^N_j)-(\pi,\Pi(p_{N*}v^N_j))] + \sum_{k=1}^{n_N} c^N_k[(\pi,\Pi(p_{N*}(h^{N,k})_*v^N_k))-(\ol{p_N},(h^{N,k})_*v^N_k)] \\
        &= 0
    \end{align*}
    
    We have shown that for any $v_f\in (T_xS)^{G_x}$ also in $\text{ker}(\varphi) = \text{ker}(\iota \circ \ol{\pi})$, it is the case that

    $$\iota(v_f)=(\pi,v_f)=\sum_i c_i[(\ol{p_i},v_i)-(\ol{q_i},h^i_*v_i)]=0$$

    \noindent in $F$. This implies that $v_f = 0$. Thus, the map $\varphi$ is injective when restricted to $(T_xS)^{G_x}$.
\end{proof}

As described above, collectively Lemmas \ref{ITS to STS lem 1} and \ref{ITS to STS lem 2} show that $\text{ker}(\varphi) = ((T_xS)^{G_x})^{\perp}$. Indeed, combined they show:

$$0 \ge \text{dim}(T_xS)^{G_x} - \text{dim(im}(\varphi))  = \text{dim(ker}(\varphi)) - \text{dim}((T_xS)^{G_x})^{\perp} \ge 0$$

\noindent with the first inequality coming from Lemma \ref{ITS to STS lem 2} and the second inequality from Lemma \ref{ITS to STS lem 1}.

Therefore, by the first isomorphism theorem $(T_xS)^{G_x} \cong F/R$, where $F/R$ represents $T_{[0]}(T_xS/G_x)$. Combined with Theorem \ref{main result cor}, Theorem \ref{ITS to STS} is attained. That is, the internal tangent space to a point $[x]\in M/G$ (thought of as a diffeological space with the quotient diffeology) is isomorphic to the stratified tangent space to $[x] \in M/G$ (thought of as a stratified space with the orbit-type stratification).

This theorem is the culminating result of this work. It provides a link between the internal tangent spaces studied here and in \cite{tan_space_and_bundle} and \cite{hector} to the well understood stratified tangent space of an orbit space.

\section[Examples of Stratified Tangent Spaces]{\textbf{V.2. Examples of Stratified Tangent Spaces}}\label{sec_strat_exs}

In light of Theorem \ref{ITS to STS}, we now determine the stratified tangent spaces to points of the orbit spaces studied above: $\rl / O(1)$ and $\rl^3 / SO(3)$. As expected in light of the above result, these stratified tangent spaces are isomorphic to the internal tangent spaces found already.

\begin{ex}
    Let $\rl / O(1)$ be the orbit space of $\rl$ under the action of $O(1)$ by scaling by $\pm 1$, as described when this space was initially introduced in Example \ref{ex r1}
    
    At the point $[0]\in \rl / O(1)$, the slice for the action in $\rl$ is the entire space $\rl$ and the stabilizer $O(1)_{0}=O(1)$ fixes only $0$ in the slice (here, the tangent action is simply the action itself restricted to elements in the stabilizer). Therefore the stratified tangent space to $[0]$ is a vector space of dimension $\text{dim}(T_xS)^{O(1)_0}=\text{dim}(\{0\})=0$. Thus, it is isomorphic to $\{0\}$.
    
    For a point $[x]\neq[0]$ in $\rl / O(1)$, the slice for the action can be taken to be any open interval containing the point $x$ but not $0$. As the stabilizer $O(1)_x$ consists only of the identity element, it fixes every point in the slice. Therefore the stratified tangent space to $[x]$ is a vector space of dimension $\text{dim}(T_xS)^{O(1)_x}=\text{dim}(T_xS)=1$. Thus, it is isomorphic to $\rl$.
\end{ex}

\begin{ex}
    Let $\rl^3 / SO(3)$ be the orbit space of $\rl^3$ under the action of $SO(3)$ by rotation, as described when this space was initially introduced in Example \ref{ex r3}
    
    At the point $[0]\in \rl^3 / SO(3)$, the slice for the action in $\rl^3$ is the entire space $\rl^3$ and the stabilizer $SO(3)_{0}=SO(3)$ fixes only $0$ in the slice (again, the tangent action is the simply the action itself restricted to elements in the stabilizer, as the action is by matrices). Therefore the stratified tangent space to $[0]$ is a vector space of dimension $\text{dim}(T_xS)^{SO(3)_0}=\text{dim}(\{0\})=0$. Thus, it is isomorphic to $\{0\}$.
    
    For a point $[x]\neq[0]$ in $\rl^3 / SO(3)$, the slice for the action can be taken to be any open interval on the unique line through both $x$ and $0$ in $\rl^3$ containing the point $x$ but not $0$. As the stabilizer $SO(3)_x$ consists only of the rotations with axis of rotation passing through $x$, and such rotations fix every point in the slice, we have $(T_xS)^{SO(3)_x}=T_xS$. Therefore the stratified tangent space to $[x]$ is a vector space of dimension $\text{dim}(T_xS)^{SO(3)_x}=\text{dim}(T_xS)=1$. Thus, it is isomorphic to $\rl$.
\end{ex}

Note that the stratified tangent spaces in both these examples were determined quickly from the slices and the tangent action at each point. This demonstrates the utility of linking the internal tangent space to a point to the stratified tangent space to said point.

\let\OLDthebibliography\thebibliography
\renewcommand\thebibliography[1]{
  \OLDthebibliography{#1}
  \setlength{\parskip}{0pt}
  \setlength{\itemsep}{0pt plus 0.3ex}
}

\startreferences
\nocite{*}
\bibliographystyle{amsplain}
\bibliography{general_final}

\end{document}